\documentclass[pdftex]{amsart} 

\usepackage{array,rotating,pdflscape,
  amsfonts,amssymb,mathtools,latexsym,upref,enumerate,cool}
\usepackage[margin=2.0cm]{geometry}
\usepackage{verbatim} 
\usepackage[english]{babel}
\usepackage[latin1]{inputenc}

\usepackage{mathtools}

\usepackage{color,tikz,pgf}
\usetikzlibrary{shapes, intersections, calc, patterns, positioning,
  matrix, arrows, fit, backgrounds, 
  decorations.pathmorphing, decorations.markings}

\newtheorem{lemma}[equation]{Lemma}
\newtheorem{prop}[equation]{Proposition}
\newtheorem{thm}[equation]{Theorem}
\newtheorem{cor}[equation]{Corollary}

\newtheorem{defn}[equation]{Definition}

\theoremstyle{definition}
\newtheorem{exmp}[equation]{Example}

\newtheorem{rmk}[equation]{Remark}

\numberwithin{equation}{section}

\newcommand{\Z}{\mathbf{Z}}

\newcommand{\Q}{\mathbf{Q}}
\newcommand{\C}{\mathbf{C}}
\newcommand{\F}{\mathbf{F}}
\newcommand{\R}{\mathbf{R}}

\newcommand{\Aut}[2]{\operatorname{Aut}_{#1}(#2)}

\newcommand{\Hom}[3][{}]{\operatorname{Hom}_{#1}(#2,#3)}
\newcommand{\Epi}[3][{}]{\operatorname{Epi}_{#1}(#2,#3)}

\newcommand{\Irr}[2]{\operatorname{Irr}_{#1}(#2)}

\newcommand{\ord}[2]{\operatorname{ord}_{#1}(#2)}

\newcommand{\SOodd}[2]{\operatorname{SO}_{#1}(#2)}

\newcommand{\GL}[3][{+}]{\operatorname{GL}^{#1}_{#2}(\F_{#3})}
\newcommand{\GU}[2]{\operatorname{GL}^-_{#1}(\F_{#2})}

\newcommand{\FGL}[2]{\operatorname{FGL}^{#1}_{#2}}
\newcommand{\FSp}[1]{\operatorname{FSp}_{#1}}

\newcommand{\OS}[2]{\operatorname{#1}_{#2}}  

\newcommand{\Sp}[2]{\operatorname{Sp}_{#1}(\F_{#2})}

\newcommand{\cat}[3]{\mathcal{#1}_{#2}^{#3}}  
  
\newcommand{\m}{morphism}
\newcommand{\pol}{polynomial}

\newcommand{\gen}[1]{{\langle}#1{\rangle}}
\newcommand{\syl}[1]{Sylow $#1$-subgroup}

\newcommand{\Mb}{M\"obius}
\newcommand{\Euc}{Euler characteristic}
\newcommand{\rchi}{\widetilde{\chi}}

\newcommand{\B}{\operatorname{B}}

\newcommand{\mynote}[1]{\noindent{\textcolor{red}{\textbf{[#1]}}}}

\newcommand{\sign}{\operatorname{sign}}

\newcommand{\SRIM}[3]{\operatorname{SRIM}^{#1}_{#2}(#3)}
\newcommand{\IM}[2]{\operatorname{IM}_{#1}(#2)}
\newcommand{\SDIM}[3]{\operatorname{SDIM}^{#1}_{#2}(#3)}
\newcommand{\MSP}[1]{\operatorname{MSP}_{#1}}

\newcommand{\Li}{\operatorname{L}}

\title{Equivariant Euler characteristics of symplectic buildings}

\author{Jesper M.~M\o ller}
\address{Institut for Matematiske Fag\\
  Universitetsparken 5\\
  DK--2100 K\o benhavn}
\email{moller@math.ku.dk}
\urladdr{htpp://www.math.ku.dk/~moller}

\usepackage[bookmarks=true,bookmarksopen=false]{hyperref}
\hypersetup{
   pdftitle = {},
   pdfauthor = {Jesper Michael MÃÂÃÂ¸ller},
   pdfpagemode = {UseOutlines},
   pdfstartview = {FitH},
   pdfborder = {0 0 0},
   backref = {true},
   colorlinks = {true},
   urlcolor = {blue},
   citecolor = {blue},
   linkcolor = {blue},
   pdftoolbar = {true},
}


\begin{document}
\date{\today}
\begin{abstract}
  The equivariant \Euc s of the buildings for the symplectic groups
  over finite fields are determined.
\end{abstract}
\maketitle
\tableofcontents

\section{Introduction}
\label{sec:introduction}


Let $G$ be a finite group, $\Pi$ a finite $G$-poset, and $r \geq 1$ a
natural number.
Atiyah and Segal \cite{atiyah&segal89} defined
the \emph{$r$th equivariant reduced \Euc\/} of the
$G$-poset $\Pi$ as the normalised  sum
\begin{equation*}
  \rchi_r(\Pi,G) = \frac{1}{|G|} \sum_{X \in \Hom{\Z^r}G} \rchi(C_{\Pi}(X(\Z^r))
\end{equation*}
of the reduced \Euc s of the $X(\Z^r)$-fixed $\Pi$-subposets,
$C_{\Pi}(X(\Z^r))$, with $X$ ranging over all homo\m s of
$\Z^r$ to $G$. (See Appendix~\ref{sec:equivariant-euc-s-1} for more
information on equivariant \Euc s.) Here are two examples of
equivariant \Euc s:
\begin{enumerate}
\item The general linear group $\GL nq$ acts on the poset $\Li_n^+(\F_q)^*$ of non-extreme subspaces of the $n$-dimensional vector space over the field $\F_q$ of prime power order $q$. The generating function
   for the $(r+1)$th equivariant reduced \Euc s of the $\GL nq$-posets $\Li_n^+(\F_q)^*$ is
\begin{equation*}
    1 + \sum_{n \geq 0} \rchi_{r+1}(\Li_n^+(\F_q), \GL nq) x^n = \prod_{0 \leq j \leq r}
  (1-q^jx)^{(-1)^{r-j}\binom{r}{j}}
\end{equation*}
according to \cite[Theorem~1.4]{jmm:eulergl+}.
\item The general unitary group $\GU nq$ acts on the poset
  $\Li_n^-(\F_q)^*$ of non-extreme
  totally isotropic subspaces of the $n$-dimensional unitary geometry
  over the field $\F_{q^2}$ of prime power order $q^2$. The generating function
   for (minus) the $(r+1)$th equivariant reduced \Euc s of the $\GU nq$-posets $\Li_n^-(\F_q)^*$ is
\begin{equation*}
    1 - \sum_{n \geq 0} \rchi_{r+1}(\Li_n^-(\F_q), \GU nq) x^n = \prod_{0 \leq j \leq r}
  (1+(-1)^{r-j}q^jx)^{(-1)^{r-j}\binom{r}{j}}
\end{equation*}
according to \cite[Theorem~1.4]{jmm:eulergl-}.
\end{enumerate}

In this paper we consider the symplectic case.
For a prime power $q$, $\Sp{2n}{q}$, the isometry group of the
symplectic $2n$-geometry, acts on the poset ${\Li}^*_{2n}(\F_q) =\{ 0
\subsetneq U \subsetneq \F_q^{2n} \mid U \subseteq U^\perp \}$ of
nonzero totally isotropic subspaces. The general definition of
equivariant \Euc s (Definition~\ref{defn:atiyahsegal}) takes in this
special case the following form.

\begin{defn}\cite{atiyah&segal89}\label{defn:chir}
  The $r$th, $r \geq 1$, equivariant reduced \Euc\ of the $\Sp{2n}{q}$-poset
  ${\Li}_{2n}^*(\F_q)$ is the normalised sum
\begin{equation*}
 \rchi_r(\Sp{2n}{q}) = \frac{1}{|\Sp{2n}{q}|}
   \sum_{X \in \Hom {\Z^r}{\Sp{2n}{q}}} \rchi(C_{\Li_{2n}^*(\F_q)}(X(\Z^r))) 
\end{equation*}
of the \Euc s of the induced subposets $C_{\Li_{2n}^*(\F_q)}(X(\Z^r))$
of
$X(\Z^r)$-invariant subspaces as $X$ ranges over all homo\m s of the
free abelian group $\Z^r$ on $r$ generators into the symplectic group.
\end{defn}

We use generating
functions or the symplectic Weyl group representation to present 
this paper's main results about
equivariant \Euc s in the symplectic case.

The generating function for the {\em negative\/} of the $r$th
equivariant reduced \Euc s
of the sequence $(\Li_{2n}^*(\F_q),\Sp{2n}q)_{n \geq 1}$
is the power series
\begin{equation}\label{eq:FSp}
  \FSp r(q,x) = 1 - \sum_{n \geq 1} \rchi_r(\Sp{2n}q)x^n
\end{equation}
with coefficients in the ring $\Z[q]$ of integral polynomials in $q$.  

\begin{thm}\label{thm:main}
  $\FSp 1(q,x)=1$ and 
  $\displaystyle
    \FSp{r+1}(q,x) = \prod_{\substack{0 \leq j \leq r \\ j \not\equiv r
        \bmod 2}} (1-q^jx)^{-\binom{r}{j}} 
  $
for all $r \geq 1$.
\end{thm}

The first generating functions $\FSp {r+1}(q,x)$ for $0 \leq r \leq 5$ are
\begin{equation*}
  1, \quad \frac{1}{1-x}, \quad
  \frac{1}{(1-qx)^2}. \quad \frac{1}{(1-x)(1-q^2x)^3}, \quad
  \frac{1}{(1-qx)^4(1-q^3x)^4},
  \quad \frac{1}{(1-x)(1-q^2x)^{10}(1-q^4x)^5}
\end{equation*}

The generating function can also be expressed in
the following alternative way. 

\begin{cor}\label{cor:expmain}
$\displaystyle \FSp{r+1}(q,x)=
\exp\Big(\sum_{n \geq 1} \frac{1}{2}  ((q^n+1)^r - (q^n-1)^r) \frac{x^n}{n}
  \Big)
$
for all $r \geq 0$.
\end{cor}

We study also
the $p$-primary
equivariant reduced \Euc s $\rchi(p, \Sp{2n}q)$ of
$(\Li_{2n}^*(\F_q), \Sp{2n}q)$ for a fixed prime $p$
(Definition~\ref{defn:pprimeuc}).  The $r$th $p$-primary generating
function, $\FSp{r}(p,q,x)$, is defined as in
\eqref{eq:FSp} except that $\rchi_r(\Sp{2n}q)$ is replaced by
$\rchi_r(p,\Sp{2n}q)$ \eqref{eq:primgenfct}.

\begin{thm}\label{thm:FSpprimary}
$\displaystyle \FSp{r+1}(p,q,x)=
\exp\Big(\sum_{n \geq 1} \frac{1}{2}  ((q^n+1)_p^r - (q^n-1)_p^r) \frac{x^n}{n}
  \Big)
$
for all $r \geq 0$.
\end{thm}

The infinite product expansions of the generating functions
\begin{alignat*}{3}
  &\FSp{r+1}(q,x) = \prod_{n \geq 1} (1-x^n)^{c_{r+1}(q,n)} &&\qquad
  &&  c_{r+1}(q,n) = \frac{1}{2n}\sum_{d \mid n} \mu(n/d) ((q^d-1)^r
  -(q^d+1)^r) \\
   &\FSp{r+1}(p,q,x) = \prod_{n \geq 1} (1-x^n)^{c_{r+1}(p,q,n)} &&\qquad
  &&  c_{r+1}(p,q,n) = \frac{1}{2n}\sum_{d \mid n} \mu(n/d) ((q^d-1)_p^r
  -(q^d+1)_p^r)
\end{alignat*}
follow immediately from the elementary
\cite[Lemma~3.7]{jmm:eulergl+}.

The ($p$-primary) equivariant reduced \Euc s are
directly linked to the structure  of the symplectic group $\Sp{2n}q$
as a finite group of Lie type.

\begin{thm} \label{thm:WSp}
  For all $n \geq 1$ and all $r \geq 0$,
  \begin{equation*}
    -\rchi_{r+1}(\Sp {2n}q) = \frac{(-1)^n}{|W(C_n)|} \sum_{w \in
      W(C_n)} \det(w) \det(q-w)^r, \ 
    -\rchi_{r+1}(p,\Sp {2n}q) = \frac{(-1)^n}{|W(C_n)|} \sum_{w \in
      W(C_n)} \det(w) \det(q-w)_p^r
  \end{equation*}
  where $W(C_n)$ is the Weyl group representation for the algebraic
  group $\operatorname{Sp}_{2n}(\overline{\F}_s)$, $s = \operatorname{char}(\F_q)$.
\end{thm}

The paper is organised as follows. In Section~\ref{sec:intro} we
briefly recall the definition of the symplectic group as the isometry
group of an even dimensional alternating bilinear form over
$\F_q$. All symplectic auto\m s have self-reciprocal characterisic
\pol s (Proposition~\ref{prop:spdual}) and
Section~\ref{sec:sympl-dual-polyn} deals with the number $\SRIM -nq$
of self-reciprocal irreducible monic \pol s of even degree $n$ over
$\F_q$
(Definition~\ref{defn:Adq}). Section~\ref{sec:equivariant-euc-s}
contains the proof of Theorem~\ref{thm:main} based upon the vanishing
result of Lemma~\ref{lemma:contrability} and the recursive relation of
(\ref{eq:recur}) which is the specific manifestation of the general
recurrence of Lemma~\ref{lemma:CgP}.  Theorem~\ref{thm:main} with
$r=1$ says that $-\rchi_2(\Sp{2n}q)=1$ for all $n \geq 1$ and all
prime powers $q$ confirming the non-block-wise form of the
Kn\"orr--Robinson conjecture for $\Sp{2n}q$ relative to the defining
characteristic (Remark~\ref{rmk:KRconj}). The $r$th $p$-primary
equivariant \Euc\ is \Euc\ computed in Morava
$K(r)$-theory. Section~\ref{sec:prim-equiv-reduc} is the $p$-primary
version of Section~\ref{sec:equivariant-euc-s}. The proof of
Theorem~\ref{thm:FSpprimary} consists in solving recurrence
\eqref{eq:FSppqrecur} which is the $p$-primary version of
\eqref{eq:recur}.  We observe that the $p$-primary equivariant \Euc\
$\rchi_r(p,\Sp{2n}q)$ for $p \nmid q$ only depends on the closure
$\overline{\langle q \rangle}$ of the subgroup
generated by $q$ in
the unit group
$\Z_p^\times$ of the $p$-adic integers.
In Section~\ref{sec:other-pres-equiv}, the equivariant \Euc s
$\rchi_{r+1}(\Sp{2n}q)$ and $\rchi_{r+1}(p,\Sp{2n}q)$ are expressed
directly in terms of integer partitions (Corollary~\ref{cor:Bnq}) or
in terms of determinants of Weyl group elements
(Theorem~\ref{thm:WSp}).
We also consider
 the reciprocal power series
   $\FSp{r+1}(q,x)^{-1}$ and $\FSp{r+1}(p,q,x)^{-1}$
   (Corollary~\ref{cor:Spinv}) and the generating functions
   $\sum_{r \geq 0} -\rchi_{r+1}(\Sp{2n}q) x^r$ with fixed parameter
   $n$ (Corollary~\ref{cor:WSp}). Example~\ref{exmp:WSp} offers
   several concrete examples of the identities established in this section.
In the short Section~\ref{sec:polyn-ident}, we formulate the symplectic
analogs of Th{\'e}venaz' \pol\
identities \cite[Theorems~A--B]{thevenaz92poly}. The paper closes with
two appendices. Appendix~\ref{sec:equivariant-euc-s-1} is a review of
basic properties of equivariant \Euc s and
Appendix~\ref{sec:comp-equiv-euc} recalls facts,
helpful for concrete calculations of equivariant \Euc s,
about Hall's eulerian
functions of groups \cite{hall36}. 

The following notation will be used in this paper:

\begin{tabular}[h]{c|l}
  $p$ & is a prime number \\
  $\nu_p(n)$ & is the $p$-adic valuation of $n$ \\
  $n_p$ & is the $p$-part of the natural number $n$ ($n_p=p^{\nu_p(n)}$) \\
  $\Z_p$ & is the ring of $p$-adic integers \\
  $q$ & is a prime power \\
  $\F_q$ & is the finite field with $q$ elements 
\end{tabular}
 

\section{The symplectic group $\Sp{2n}q$}
\label{sec:intro}

Let $q$ be a prime power and $n \geq 1$ a natural number. The
symplectic $2n$-geometry is the vector space
$V_{2n}(\F_{q}) = \F_{q}^{2n}$ of dimension $2n$ over the field
$\F_{q}$ equipped with the non-degenerate \cite[Definition
3.1]{artin57} alternating ($\gen{u,v} = -\gen{v,u}$) bilinear form
given by
\begin{equation}\label{eq:spform}
  \gen{u,v} = uJv^t = \sum u_iv_{-i}-\sum u_{-i}v_{i}, \qquad
J =
  \begin{pmatrix}
    0 & E \\ -E & 0
  \end{pmatrix}, \qquad u, v \in V_{2n}(\F_{q})
\end{equation}
for all
$u=(u_1,\ldots,u_n,u_{-1},\ldots,u_{-n}),
v=(v_1,\ldots,v_n,v_{-1},\ldots,v_{-n}) \in V_{2n}(\F_{q})$.  The
symplectic $2n$-geometry is the orthogonal direct sum,
$\gen{e_1,e_{-1}} \perp \cdots \perp \gen{e_n,e_{-n}}$, of the $n$
hyperbolic planes $\gen{e_i,e_{-i}}$, $1 \leq i \leq n$ \cite[Theorem
3.7]{artin57}.  The symplectic group
$\Sp{2n}{q} = \{ g \in \GL nq \mid gJg^t=J \}$ is the group of all
auto\m s of the symplectic $2n$-geometry \cite[\S 2.7]{GLSIII}. The
center of $\Sp{2n}{q}$ is trivial if $q$ is even and of order $2$ if
$q$ is odd \cite[Theorem 5.2]{artin57}.

A subspace $U$ of the symplectic geometry $(V_{2n}(\F_{q}), \gen{\cdot,\cdot})$
is totally isotropic if $\gen{U,U}=0$.  The symplectic group acts on the
poset $\Li_{2n}^*(\F_q)$ of all nontrivial totally isotropic
subspaces. Since all vectors are isotropic, $\gen{u,u}=0$, all
$1$-dimensional subspaces are in $\Li_{2n}^*(\F_q)$ (and $\Li_2^*(\F_q)$
is simply the set of $1$-dimensional subspaces of $V_2(\F_q)$).

When the prime power $q=2^e$ is even,
$\Sp{2n}q \cong \SOodd{2n+1}{\F_q}$ \cite[Theorem 2.2.10]{GLSIII}.



\section{Characteristic polynomials of symplectic auto\m s}
\label{sec:sympl-dual-polyn}

\begin{defn}\cite[Definition 3.12]{lidlnieder97}\label{defn:spdual}
  The reciprocal of a degree $n$ \pol\
  $p(x)=a_0x^n+a_{1}x^{n-1}+\cdots+a_{n-1}x+a_n$ over $\F_q$ with
  nonzero constant term is the degree $n$ \pol\
  $p^*(x)=a_0+a_{1}x+\cdots+a_{n-1}x^{n-1}+a_nx^n = x^np(x^{-1})$.
  The \pol\ $p$ is self-reciprocal if $p^*=p$.
\end{defn}

The operation $p(x) \to p^*(x)$ is involutory, multiplicative, and
divisibility respecting ($p^{**}(x)=p(x)$,
$(p_1(x)p_2(x))^* = p_1^*(x)p_2^*(x)$,
$p_1 \mid p_2 \implies p_1^* \mid p_2^*$) on the set of \pol s
$p(x) \in \F_q[x]$ with $p(0) \neq 0$.  The multisets of roots for a
\pol\ and its reciprocal correspond under the inversion map
$\bar\F_q^\times \to \bar\F_q^\times \colon \alpha \to \alpha^{-1}$.
The \pol\ $p(x)=a_0x^n+a_{1}x^{n-1}+\cdots+a_{n-1}x+a_n$ is
self-reciprocal if and only it has a palindromic coefficient sequence,
$a_i=a_{n-i}$, $0 \leq i \leq n$.

\begin{prop} \label{prop:spdual} The characteristic \pol\ of any
  symplectic auto\m\ $g \in \Sp{2n}q$ is a self-reciprocal monic
  \pol .
\end{prop}
\begin{proof}
Let $c_g$ denote the  characteristic \pol\ of $g$ and $p(x)$ a monic
\pol\ of degree $n$. The relation
\begin{equation}\label{eq:spdual}
  \forall u,v \in V \colon \gen{up(g),v} =  \gen{ug^n,v  p^*(g)}
\end{equation}
implies that
  $r \nmid c_g \iff \text{$r(g)$ is invertible} \iff \text{$r^*(g)$ is
    invertible} \iff r^* \nmid c_g $
for any monic irreducible \pol\ $r$ over $\F_q$, and hence that
$c_g^* = c_g$. (Alternatively, taking for given that $\det(g)=+1$, we get $c_g(\lambda) =
c_{g^t}(\lambda) = \det(g^t - \lambda I) = \det(Jg^{-1}J^{-1} -
\lambda I) = \det(-Jg^{-1}J + J(\lambda I)J) = \det(-g^{-1} + \lambda I) = \det(\lambda
g-I) = \lambda^{2n}\det(g - I/\lambda) = \lambda^{2n} c_g(1/\lambda)$.)
\end{proof}

 Conversely, any self-reciprocal \pol\ is the
 characteristic \pol\ for a symplectic auto\m\ \cite[Theorem A.1]{rivin2008}.
 


\begin{prop}\label{prop:qnsr}
  The number of self-reciprocal monic \pol s of even degree $2n$ is
  $q^n$.
\end{prop}
\begin{proof}
  Self-reciprocal monic \pol s of degree $2n$ have
  palindromic coefficients.
  %
\end{proof}

\begin{lemma}\label{lemma:fixmonrec}
  The transformation $r(x) \to r^*(x)/r(0)$ is an involution on the
  set of irreducible monic \pol s $r(x) \in \F_q[x]$  with $r(0)  \neq 0$.

  The irreducible monic \pol\ $r(x)$ with $r(0) \neq 0$ is fixed under
  this involution, when $\deg r =1$, if and only if $r(x) = x \pm 1$,
  and when $\deg r > 1$, if and only if $r(x)$ has even degree,
  $r(0)=1$, and $r(x)$ is self-reciprocal.
\end{lemma}
\begin{proof}
  If $r(x)$ has degree at least $2$, the degree of $r$ must be even,
  since the set of roots in the algebraic closure is invariant under
  inversion
  $\bar\F_q^\times \to \bar\F_q^\times \colon \alpha \to \alpha^{-1}$
  and the fixed points, $\pm 1$, are not roots of $r(x)$. The relation
  $r(0)r(x)=r^*(x)=x^{\deg r}r(1/x)$ evaluated at $x=-1$ gives
  $r(0)=1$. Thus $r(x)=r^*(x)$ and $r(x)$ is self-reciprocal.
\end{proof}

\begin{prop}\label{prop:srfact}
  Let $p(x)$ be a self-reciprocal monic \pol . The canonical
  factorisation \cite[Theorem 1.59]{lidlnieder97} of $p(x)$ has the
  form
  \begin{equation*}
    p(x) =
    \begin{cases} \displaystyle
    (x-1)^{a_-}  \times (x+1)^{a_+}
    \times \prod_i r_i^-(x)^{m_i^-} \times \prod_j
    (s_j(x)s_j^*(x)/s_j(0))^{m_j^+} & \text{$q$ odd} \\
    \displaystyle  (x+1)^{a_+}
    \times \prod_i r_i(x)^{m_i^-} \times \prod_j
    (s_j(x)s_j^*(x)/s_j(0))^{m_j^+} & \text{$q$ even}
  \end{cases}
    \end{equation*}
  where
  \begin{equation*}
    \deg p =
    \begin{cases}
      a_-+a_++\sum_im_i^-\deg r_i + 2\sum_j m_j^+\deg s_j  &
      \text{$q$ odd} \\
      a_++\sum_im_i^-\deg r_i + 2\sum_j m_j^+\deg s_j  &
      \text{$q$ even}
    \end{cases}
  \end{equation*}
  and $a_-,a_+, m^-_i, m^+_j \geq 0$, $a_-$ is even, the $r_i^-(x)$ are
  self-reciprocal irreducible monic \pol s of even degree at least
  $2$, and the $s_j(x)$ are non-self-reciprocal irreducible monic \pol s
  distinct from $x -1$. Conversely, any \pol\ with a canonical
  factorisation of this form is a self-reciprocal monic \pol .
\end{prop}
\begin{proof}
  Let $p(x) = \prod r_i(x)^{e_i}$ be the canonical
  factorisation. Since $p(x)$ is monic and self-reciprocal, $p(0)=1$,
  and $p(x) = p^*(x)= p^*(x)/p(0) = \prod (r_i^*(x)/r_i(0))^{e_i}$
  where $r_i^*(x)/r_i(0)$ are irreducible monic \pol s \cite[Remark
  2.1.49]{mullen}. Thus the multiset of the irreducible factors of
  $p(x)$ is invariant under the involution $r(x) \longleftrightarrow
  r^*(x)/r(0)$. Group the
irreducible factors into those fixed by this involution and pairs
  interchanged by it. An irreducible factor of degree $\geq 2$ is
  fixed by the involution if and only if it self-reciprocal according to
Lemma~\ref{lemma:fixmonrec}. Any irreducible linear factor, which has the
form $x-\alpha$ for some $\alpha \in \F_q^\times$, is fixed by the
involution if and only if $\alpha = \pm 1$ ($\alpha =1$ when $q$ is
even). Thus $p(x)$ has a canonical factorisation of the form shown in
the proposition. When $q$ is odd, the multiplicity, $a_-$, of 
the factor $x-1$ is even because
$1 = p(0) = (-1)^{a_-}$.

Conversely, if $p(x)$ has a factorisation as in the proposition, then
$(x-1)^{a_-}$ is self-reciprocal as $a_-$ is even, and as also the other
factors, $x+1$, $r_i^-(x)$, $s_j(x)s_j^*(x)/s_j(0)$, are
self-reciprocal, the \pol\ $p(x)$ is self-reciprocal.
\end{proof}

All factors on the right hand side of the formula of
Proposition~\ref{prop:srfact} are self-reciprocal. The exponent $a_-$
is even while $a_+$ has the same parity as the degree of $p$.

\begin{defn}\cite[Definition 2.1.23, Remark 3.1.19]{mullen}\label{defn:Adq}
  For every integer $n \geq 1$
  \begin{itemize}
  \item $\IM nq$ is the number of Irreducible Monic \pol s $p(x)$ of degree
    $n$ over $\F_{q}$ with $p(0) \neq 0$ 
\item $\SRIM -nq$  is the number of Self-Reciprocal Irreducible Monic
  \pol s $p(x)$ of even
  degree $2n$ over $\F_{q}$  
\item $\SRIM  +nq$  is the number of unordered pairs $\{ p(x), p^*(x)/p(0) \}$ of irreducible monic
  \pol s $p(x)$ of
  degree $n$ over $\F_{q}$ with $p(0) \neq 0$ and $p(x) \neq p^*(x)/p(0)$
  \end{itemize}
\end{defn}

For any $n \geq 1$ \cite[Theorem 2.1.24, Theorem 3.1.20]{mullen}
\cite[Theorem 3]{meyn1990}
\begin{equation}\label{eq:SRIM} 
  \IM nq = \frac{1}{n} \sum_{d \mid n} \mu(d)(q^{n/d}-1)
  \qquad
  \SRIM -{n}q =
  \begin{cases} \displaystyle
    \frac{1}{2n} \sum_{\substack{d \mid n \\ d \equiv 1 \bmod 2}} 
      \mu(d)(q^{n/d}-1) & \text{$q$ odd} \\ \displaystyle
       \frac{1}{2n} \sum_{\substack{d \mid n \\ d \equiv 1 \bmod 2}} 
      \mu(d)q^{n/d} & \text{$q$ even} 
  \end{cases}
\end{equation}
and we have
\begin{equation}\label{eq:IMSRIM}
  \IM nq =
  \begin{cases}
    2 \SRIM +nq & \text{$n>1$ odd} \\
    2 \SRIM +nq + \SRIM -{n/2}q & \text{$n>1$ even} \\
  \end{cases}
\end{equation}
In degree $n=1$,  in particular, $\IM 1q = q-1$ and
\begin{equation*}
  \SRIM +1q =
  \begin{cases}
    \frac{1}{2}(q-3) & \text{$q$ odd} \\
    \frac{1}{2}(q-2) & \text{$q$ even} 
  \end{cases}
\end{equation*}
For odd $q$, the $\frac{1}{2}(q-3)$ unordered pairs are the pairs
$\{x-\alpha,x-\alpha^{-1}\}$ with $\alpha \in \F_q^* - \{-1,+1\}$. For
even $q$, the $\frac{1}{2}(q-2)$ unordered pairs are the pairs
$\{x-\alpha,x-\alpha^{-1}\}$ with $\alpha \in \F_q^* - \{1\}$.

\begin{lemma}\label{lemma:SRIMm}
  Let $m \geq 1$ and $k \geq 0$. Then
  $2^{k}\SRIM -{2^km}q = \SRIM -m{q^{2^k}}$ for all prime powers $q$.
  When $m$ is odd,
  \begin{equation*}
    2^{k+1}\SRIM -{2^km}q  \stackrel{\text{$q$ odd}}{=} \IM m{q^{2^k}} \qquad
    2^{k+1}\SRIM -{2^km}q \stackrel{\text{$q$ even}}{=}
    \begin{cases}
      \IM 1q +1 & m=1 \\ \IM m{q^{2^k}} & m>1
    \end{cases}
  \end{equation*}
\end{lemma}
\begin{proof}[Proof for odd $q$]
  When computing $\SRIM -{2^km}q$ from formula \eqref{eq:SRIM}, only divisors of $m$ matter so
  \begin{equation*}
    2^k \SRIM -{2^km}q =
    \frac{1}{2m} \sum_{d \mid m} \mu(d)(q^{2^km/d}-1) =
    \SRIM -m{q^{2^k}}
  \end{equation*}
  Assuming $m$ is odd, $\SRIM -m{q} = \frac{1}{2} \IM m{q}$, so
  $2^k \SRIM -{2^km}q = \SRIM -m{q^{2^k}} = \frac{1}{2}\IM m{q^{2^k}}$.
\end{proof}

\begin{lemma}\label{lemma:IMSRIM}
  For all $n \geq 1$,
  \begin{equation*}
    \IM nq  \stackrel{\text{$q$ odd}}{=}
    \begin{cases}
      2 \SRIM -nq & \text{$n$ odd} \\
      2 \SRIM -nq - \SRIM -{n/2}q & \text{$n$ even} \\
    \end{cases} \qquad
    \IM nq  \stackrel{\text{$q$ even}}{=}
    \begin{cases}
      2\SRIM -1q -1 & n=1 \\
      2 \SRIM -nq & \text{$n>1$ odd} \\
      2 \SRIM -nq - \SRIM -{n/2}q & \text{$n$ even} \\
    \end{cases}
  \end{equation*}
\end{lemma}
\begin{proof}[Proof for odd $q$]
  Let $m \geq 1$ be odd. By Lemma~\ref{lemma:SRIMm},
  $\IM mq =2 \SRIM -mq$ and
  \begin{multline*}
    \IM{2^km}q =
    \frac{1}{2^km} \big(
    \sum_{d \mid m} \mu(d)(q^{2^km/d}-1) -
    \sum_{d \mid m} \mu(d)(q^{2^{k-1}m/d}-1) \big) =
    \frac{1}{2^k} \big( \IM m{q^{2^k}} - \IM m{q^{2^{k-1}}} \big) \\
    \stackrel{\text{Lemma~\ref{lemma:SRIMm}}}{=}
    2\SRIM -{2^km}q - \SRIM -{2^{k-1}m}q 
  \end{multline*}
  for $k \geq 1$.
\end{proof}

\begin{lemma}\label{lemma:Sm+Sp}
  For all $n \geq 1$,
  \begin{equation*}
    \SRIM -nq + \SRIM +nq \stackrel{\text{$q$ odd}}{=}
    \begin{cases}
      \IM 1q -1 & n=1 \\ \IM nq & n>1
    \end{cases} \qquad
    \SRIM -nq + \SRIM +nq \stackrel{\text{$q$ even}}{=} \IM nq, \quad
    n \geq 1
  \end{equation*}
\end{lemma}
\begin{proof}
  Assume the prime power $q$ is odd. If $n=1$ then
  $\IM 1q = q-1$,
  $\SRIM -1q=\frac{1}{2}(q-1)$, and $\SRIM +1q=\frac{1}{2}(q-3)$ so
  that indeed $\SRIM -1q + \SRIM +1q =q-2 =\IM 1q -1$. For odd $m>1$,
  $\SRIM -mq = \frac{1}{2} \IM mq =\SRIM +mq$ so
  that clearly $\SRIM -mq + \SRIM +mq = \IM mq$.
  For odd $m \geq 1$ and $k \geq 1$,
  \begin{multline*}
    \SRIM -{2^km}q + \SRIM +{2^km}q \stackrel{\eqref{eq:SRIM}}{=}
    \SRIM -{2^km}q + \frac{1}{2} \big( \IM{2^km}q - \SRIM -{2^{k-1}m}q
    \big) \\=
    \frac{1}{2} \big( 2 \SRIM -{2^km}q - \SRIM -{2^{k-1}m}q \big) +
    \frac{1}{2} \IM{2^km}q
    \stackrel{\text{Lemma~\ref{lemma:IMSRIM}}}{=}
    \frac{1}{2} \IM{2^km}q + \frac{1}{2} \IM{2^km}q =\IM{2^km}q 
   \end{multline*}
   This finishes the proof for odd $q$.

   Assume that $q$ is an even prime power. In degree $1$,
   $\SRIM -1q + \SRIM +1q = \frac{1}{2}q + \frac{1}{2}q-1 =q-1 = \IM
   1q$.  For odd $m >1$, $\SRIM -mq = \frac{1}{2} \IM mq = \SRIM +mq$
   by Lemma~\ref{lemma:SRIMm} so $\SRIM -mq + \SRIM +mq = \IM mq$.
   For odd $m \geq 1$ and any $k \geq 1$, we can again use
   Lemma~\ref{lemma:IMSRIM} and it follows, as for odd $q$, that
   $\SRIM -{2^km}q + \SRIM -{2^km}q =\IM{2^km}q$.
\end{proof}

\begin{figure}[t]
  \centering
  \begin{tabular}[h]{>{$}c<{$}|*{7}{>{$}c<{$}}}
    {} & n=1 & n=2  & n=3 & n=4 & n=5 & n=6 & n=7  \\ \hline
    \SRIM -nq & \frac{1}{2}(q - 1)  & \frac{1}{4}(q^2 - 1) & 
    \frac{1}{6}(q^3 - q) & \frac{1}{8}(q^4-1) & \frac{1}{10}(q^5-q)
     & \frac{1}{12}(q^6-q^2)  & \frac{1}{14}(q^7-q)\\
    \SRIM +nq & \frac{1}{2}(q - 3) &
    \frac{1}{4}(q^2 - 2q + 1) &
    \frac{1}{6}(q^3 - q) &
    \frac{1}{8}(q^4 - 2q^2 + 1) &
    \frac{1}{10}(q^5 - q) &
    \frac{1}{12}(q^6 - 2q^3 - q^2 + 2q) &
    \frac{1}{14}(q^7 - q) 
    \end{tabular} 
  \caption{The \pol s $\SRIM {\pm}nq$ for odd $q$}
  \label{fig:SRIM}
\end{figure}


\section{Proofs of Theorem~\ref{thm:main} and Corollary~\ref{cor:expmain} }
\label{sec:equivariant-euc-s}

We recursively compute the generating functions $\FSp r(q,x)$,
$r \geq 1$, of Theorem~\ref{thm:main}.

\begin{lemma}\label{lemma:rchi1}
  $-\rchi_1(\Sp{2n}q) = 0$ for all $n \geq 1$.
\end{lemma}
\begin{proof}
  The first equivariant reduced \Euc\ is the usual reduced \Euc\ of
  the orbit space $|\Li_{2n}^*(\F_q)|/\Sp{2n}q$ for the action of
  $\Sp{2n}q$ on its building (Lemma~\ref{lemma:chi1}).
  Webb's theorem \cite[Proposition~8.2.(i)]{webb87} says that the
  reduced \Euc\ of this orbit space equals $0$.
\end{proof}


\begin{lemma}\cite{quillen78}\label{lemma:contrability}
  If $A$ is an abelian subgroup of $\Sp{2n}q$ and $\gcd(|A|,q)>1$, then
 $\rchi_r(C_{\Li_{2n}^*(\F_q)}(A),C_{\Sp{2n}q}(A))=0$ for all $r
   \geq 1$.
\end{lemma}
\begin{proof}
  It suffices to show that $\rchi(C_{\Li_{2n}^*(\F_q)}(A))=0$ for
  any abelian subgroup $A$ of $\Sp{2n}q$ with $O_s(A) \neq 1$ where $s$ is
  the characteristic of the field $\F_q$. We may replace the poset
  $\Li^*_{2n}(\F_q)$ by the poset $\cat S{\Sp{2n}q}{s+*}$ of
  non-trivial $s$-subgroups of $\Sp{2n}q$
  \cite[Theorem~3.1]{quillen78}. The fixed poset $C_{\cat
    S{\Sp{2n}q}{s+*}}(A)$ admits the conical contraction $B
  \leq BO_s(A) \geq O_s(A)$ defined for all $B \in C_{\cat
    S{\Sp{2n}q}{s+*}}(A)$.
\end{proof}

\begin{lemma} 
  \label{lemma:recur}
  For $n \geq  1$ and $r \geq 1$, the $(r+1)$th equivariant \Euc\ of the $\Sp{2n}q$-poset
  $\Li_{2n}^*(\F_q)$ is
  \begin{equation*}
    \rchi_{r+1}(\Sp {2n}q) = \sum_{\substack{X \in \Hom{\Z}{\Sp{2n}q}/\Sp{2n}q \\ \GCD{q, |X(\Z)|}=1}}
    \rchi_{r}(C_{\Li_{2n}^*(\F_q)}(X), C_{\Sp{2n}q}(X))
  \end{equation*}
  where the sum ranges over semisimple conjugacy classes in $\Sp{2n}q$.
\end{lemma}
\begin{proof}
  This is a special case of the general formula 
  from Lemma~\ref{lemma:CgP}. By
  Lemma~\ref{lemma:contrability}, we need only the conjugacy classes
  of order prime to $q$ (semisimple classes).
\end{proof}

The centraliser of the semisimple element $g$ of $\Sp {2n}{q}$ with
characteristic \pol\ as in Proposition~\ref{prop:srfact} is
 \cite{fong_srinivasan89} \cite[(3.3)]{vinroot2017}
\begin{equation}\label{eq:vinroot}
  C_{\Sp{2n}q}(g) = \Sp{a_-}{q} \times \Sp{a_+}{q} \times
    \prod_i \GU{m_i^-}{{q^{\frac{1}{2}d_i^-}}} \times \prod_j \GL{m_j^+}{{q^{d_j}}}
  \end{equation}
  and the contribution,
  $-\rchi_{r}(C_{\Li_{2n}(\F_q)^*}(g), C_{\Sp{2n}q}(g))$, to the sum
  $-\rchi_{r+1}(\Sp {2n}q)$ of Lemma~\ref{lemma:recur} from $g$ is
 \begin{equation*}
    -\rchi_r(\Sp{a^-}{q}) \times -\rchi_r(\Sp{a^+}{q}) \times
    \prod_i -\rchi_r(\GU{m_i^-}{{q^{\frac{1}{2}d_i^-}}}) \times
    \prod_j \rchi_r(\GL{m_j^+}{{q^{d_j}}})
  \end{equation*}
  with a sign change similar to that of
  \cite[Lemma~4.3]{jmm:eulergl-}.  The characteristic \pol\ induces a
  bijection between the set of semisimple classes in $\Sp{2n}{q}$ and
  the set of self-reciprocal \pol s of degree $2n$
  \cite[\S3.1]{vinroot2017} \cite{wall63}. We conclude from these
  facts that $-\rchi_{r+1}(\Sp{2n}q)$ equals
  \begin{equation}\label{eq:basicrec}
    \begin{aligned}
      \sum_{(a^-,a^+,\lambda^-,\lambda^+) } &
    (-\rchi_r(\Sp{2a^-}q))(- \rchi_r(\Sp{2a^+}q)) \times\\
  \prod_{\substack{d^- \\ \exists m^- \colon (m^-,d^-)^{e(m^-,d^-)}
      \in \lambda^- }} &
  \binom{\SRIM -{d^-}q}{[e(m^-,d^-) \colon (m^-,d^-)^{e(m^-,d^-)} \in \lambda^-]}
  \prod_{(m^-,d^-)^{e(m^-,d^-)} \in \lambda^-}
  -\rchi_r(\GU{m^-}{q^{\frac{1}{2}d^-}})^{e(m^-,d^-)}  \times \\
  \prod_{\substack{d^+  \\ \exists m^+ \colon (m^+,d^+)^{e(m^+,d^+)} \in \lambda^+}}
  &\binom{\SRIM +{d^+}q}{[e(m^+,d^+) \colon (m^+,d^+)^{e(m^+,d^+)}
    \in \lambda^+]}
  \prod_{(m^+,d^+)^{e(m^+,d^+)} \in \lambda^+}
  \rchi_r(\GL{m^+}{q^{d^+}})^{e(m^+,d^+)} 
    \end{aligned}
  \end{equation}
where the sum runs over all $(a^-,a^+,\lambda^-,\lambda^+)$ where
$a^\pm$ are positive integers,
$\lambda^\pm = \{(m_i^\pm,d_i^\pm)^{e_i^{\pm}}\}$ are multisets of
pairs of positive integers such that
$a^-+a^++\sum m_i^-d_i^-e_i^- + \sum m_j^+d_j^+e_j^+=n$ and the
$d_i^-$ are even. 

We are here using multinomial coefficients as defined below. 

\begin{defn}\label{defn:mnc}
For a rational \pol\ $m \in \Q[q]$  and $k_1,\ldots,k_s \geq 0$ a finite
sequence of nonnegative integers,
define the multinomial coefficients to be
\begin{align*}
  &\binom m{k_1,\ldots,k_s} =
    \frac{m(m-1) \cdots (m+1-\sum k_i)}{k_1! \cdots k_s! } =
    \begin{cases}
     \frac{m!}{k_1! \cdots k_s! \cdot (m-\sum k_i)!}
      & \sum k_i \leq m \\
      0 & \sum k_i > m
    \end{cases}
  \\
  & \binom m{-k_1,\ldots,-k_s}  
  = (-1)^{k_1+\cdots+k_s} \binom m{k_1,\ldots,k_s} \\
  &\binom {-m}{-k_1,\ldots,-k_s} =
    \frac{m(m+1) \cdots (m-1+\sum k_i)}{k_1! \cdots k_s!} = \binom{m-1+\sum k_i}{k_1,\ldots,k_s}
\end{align*}
\end{defn}

Equation~\eqref{eq:recur} below uses $T_S$-transformed generating
functions (Definition~\ref{defn:TSa}) to succinctly express the recurrence
relation for $\rchi_{r+1}(\Sp{2n}q)$.

\begin{defn}\label{defn:MSp}
  $M_n=\{\lambda\}$, $n \geq 0$, is the set of all 
   multisets of pairs of natural numbers
  \begin{equation*}
   \lambda=\{(m_1,d_1)^{e(m_1,d_1)},  \ldots, (m_t,d_t)^{e(m_t,d_t)}\}
  \end{equation*}
  with $\sum_i m_id_ie(m_i,d_i)=n$.
\end{defn}

The first of the sets $M_n$ are
 \begin{align*}
   M_0 &= \emptyset, \quad
   M_1 = \{ \{(1,1)\} \}, \quad
   M_2 = \{ 
    \{ (1, 2) \},
    \{ (2, 1) \},
         \{ (1, 1)^2 \} \\
   M_3 &= \{
         \{ (1, 3) \},
    \{ (3, 1) \},
    \{ (1, 1), (1, 2) \},
    \{ (1, 1), (2, 1) \},
         \{ (1, 1)^3 \} \} \\
   M_4 &= \{
          \{ (1, 1), (1, 3) \},
         \{ (1, 1)^4 \},
         \{ (1, 4) \},
    \{ (1, 1), (3, 1) \},
    \{ (1, 1)^2, (1, 2) \},
    \{ (1, 2)^2 \}, \\
    &\qquad\{ (2, 1)^2 \},
    \{ (2, 2) \},
    \{ (1, 1)^2, (2, 1) \},
    \{ (1, 2), (2, 1) \},
    \{ (4, 1) \} \}
 \end{align*}
 

 Let $S=(S(n))_{n \geq 1}$ and $a=(a(n))_{n \geq 1}$ by sequences of
 integral \pol s $S(n),a(n) \in \Z[q]$.
 
  \begin{defn}\cite[Definition 3.1]{jmm:eulergl+}\label{defn:TSa}
    The $S$-transform  of $a$, $T_S(a)$,  is the \pol\ sequence with
   \begin{equation*}
      T_S(a)(n)(q) =           
  \sum_{\lambda \in M_n}  
  \prod_{\{d \mid  \exists m \colon (m,d) \in B(\lambda) \}}
  \binom{S(d)(q)}{[E(\lambda,(m,d)) \colon (m,d) \in
    B(\lambda)]}  \prod_{(m,d) \in B(\lambda)} a(m)(q^d)^{E(\lambda,(m,d))}
   \end{equation*}
as its $n$th term for every $n \geq 1$.
\end{defn}

Using the concept of $T_S$-transforms we may express Lemma~\ref{lemma:recur} or
\eqref{eq:basicrec} by the
recurrence
\begin{equation}
  \label{eq:recur}
  \FSp {r+1}(x) =
  \begin{cases}
\FSp {r}(x)^2
  T_{\SRIM -{}q}(\FGL{-}r(x))
  T_{\SRIM +{}q}(\FGL{+}r(x)) & \text{$q$ odd} \\
  \FSp {r}(x)
  T_{\SRIM -{}q}(\FGL{-}r(x))
  T_{\SRIM +{}q}(\FGL{+}r(x)) & \text{$q$ even} 
  \end{cases}
\qquad (r \geq 1)    
\end{equation}
where the generating functions $\FGL{\pm}r(x)$ of
\cite[(1.3)]{jmm:eulergl+} and \cite[(1.2)]{jmm:eulergl-} have been
transformed relative to the \pol\ 
sequences $(\SRIM {\pm}dq)_{d \geq 1}$. We now start the computation
of the product of these two transformed generating functions. First, a
well-known lemma:

\begin{lemma}\cite[ p 258]{wall99}\label{lemma:IM}
  $\displaystyle T_{\IM {}q}(1-x) = \frac{1-qx}{1-x}$.
\end{lemma}
\begin{proof}
  The $\IM {}q$-transform of $1-x$ is
  \begin{equation*}
  T_{\IM{}q}(1-x) \stackrel{\text{\cite[(3.2)]{jmm:eulergl+}}}{=}
\prod_{n \geq 1} (1-x^n)^{\IM nq} \stackrel{\text{\cite[Lemma
    5.30]{jmm:eulergl+}}}{=} \exp\big( \sum_{n \geq 1} (1-q^n)x^n/n \big) =
    \frac{1-qx}{1-x}
  \end{equation*}
  where we use the \Mb\ inverse, $q^n-1 = \sum_{d \mid n} \IM dq$, of
  the left equation of \eqref{eq:SRIM}.
\end{proof}


\begin{lemma}\label{lemma:ofir} 
  $\displaystyle
    T_{\SRIM -{}q}(1-x) T_{\SRIM +{}q}(1-x)=
    \begin{cases}
      \frac {1-qx}{(1-x)^2} & \text{$q$ odd} \\
       \frac{1-qx}{1-x} & \text{$q$ even} 
    \end{cases} $ and 
$\displaystyle
    T_{\SRIM -{}q}(1+x) T_{\SRIM +{}q}(1-x) = \frac{1}{1-x}
  $.
\end{lemma}
\begin{proof}
  Thanks to  the  identity of  Lemma~\ref{lemma:IM} it is easy to determine
  \begin{equation*}
    T_{\SRIM -{}q}(1-x) T_{\SRIM +{}q}(1-x)= T_{\SRIM -{}q + \SRIM +{}q}(1-x)
    \stackrel{\text{Lemma~\ref{lemma:Sm+Sp}}}{=}
    \begin{cases}
      (1-x)^{-1}T_{\IM {}q}(1-x) = \frac{1-qx}{(1-x)^2} & \text{$q$ odd} \\
      T_{\IM {}q}(1-x) = \frac{1-qx}{1-x} & \text{$q$ even} 
    \end{cases}
  \end{equation*}
  Observe that
  \begin{equation}
    \label{eq:S^-}
    T_{\SRIM -{}q} \Big(
      \frac{1-x}{1+x} \Big) =
      \begin{cases}
        \frac{1-qx}{1-x}  & 2 \nmid q \\
        1-qx & 2 \mid q
      \end{cases}
    \end{equation}
    because, as $1+x = \frac{1-x^2}{1-x}$,
    \begin{multline*}
      T_{\SRIM -{}q} \Big(
      \frac{1-x}{1+x} \Big) =
    \frac {T_{\SRIM -{}q}(1-x)^2}{T_{\SRIM -{}q}(1-x^2)} =
    \frac{\prod_m (1-x^m)^{2\SRIM -mq}}{\prod_m (1-x^{2m})^{\SRIM
        -mq}} \\=
     \prod_{ 2 \nmid m} (1-x^m)^{2\SRIM -mq} \times
     \prod_{ 2 \mid m} (1-x^m)^{2\SRIM -mq- \SRIM -{m/2}q} 
     \stackrel{\text{Lemma~\ref{lemma:IMSRIM}}}{=}
     \begin{cases}
       T_{\IM {}q}(1-x) = (1-x)^{-1}(1-qx) & \text{$q$ odd} \\
       (1-x)T_{\IM {}q}(1-x) =1-qx & \text{$q$ even}
     \end{cases}
   \end{multline*}
   Therefore
   \begin{equation*}
     T_{\SRIM -{}q}(1+x) T_{\SRIM +{}q}(1-x) =
      T_{\SRIM -{}q} \Big(
      \frac{1+x}{1-x} \Big) T_{\SRIM  -{}q}(1-x) T_{\SRIM  +{}q}(1-x)
      =\frac{1}{1-x}
      \end{equation*}
  for all prime powers $q$.
\end{proof}

From recurrence \eqref{eq:recur} and Lemma~\ref{lemma:ofir}, we get that the second
generating function is
\begin{equation*}
  \FSp 2(x) = T_{\SRIM -{}q}(\FGL{-}1(x)) T_{\SRIM +{}q}(\FGL{+}1(x)) =
  T_{\SRIM -{}q}(1+x) T_{\SRIM +{}q}(1-x) =
  \frac{1}{1-x}
\end{equation*}
as the first generating function is $\FSp 1(x)=1$ by Lemma~\ref{lemma:rchi1}.

\begin{lemma}\label{lemma:solrecur}  
  For all $r \geq 1$
  \begin{equation*}
    T_{\SRIM -{}q}(\FGL -{r+1}(x)) T_{\SRIM +{}q}(\FGL +{r+1}(x))
    \prod_{\substack{0 \leq j \leq r+1 \\ j \equiv r
        \bmod 2}} (1-q^jx)^{\binom{r+1}{j}} =
    \begin{cases}
      \displaystyle
      \prod_{\substack{0 \leq j \leq r \\ j \not\equiv r
          \bmod 2}} (1-q^jx)^{2\binom{r}{j}} & \text{$q$ odd} \\
      \displaystyle
      \prod_{\substack{0 \leq j \leq r \\ j \not\equiv r
          \bmod 2}} (1-q^jx)^{\binom{r}{j}} & \text{$q$ even} 
    \end{cases}
\end{equation*}
\end{lemma}  
\begin{proof}
  With the formulas for $\FGL{\pm}{r+1}(x)$ from \cite[Theorem
  1.4]{jmm:eulergl+} and \cite[Theorem 1.3]{jmm:eulergl-} as input we
  compute for odd $q$ that 
  \begin{multline*}
    T_{\SRIM -{}q}(\FGL -{r+1}(x)) T_{\SRIM +{}q}(\FGL +{r+1}(x)) = \\ 
    \frac{\displaystyle \prod_{\substack{0 \leq j \leq r \\ j \equiv r
          \bmod 2}}T_{\SRIM -{}q}(1+q^jx)^{\binom{r}{j}}}
    {\displaystyle \prod_{\substack{0 \leq j \leq r \\ j \not\equiv r
          \bmod 2}}T_{\SRIM -{}q}(1-q^jx)^{\binom{r}{j}}}
      \frac{\displaystyle \prod_{\substack{0 \leq j \leq r \\ j \equiv r
          \bmod 2}}T_{\SRIM +{}q}(1-q^jx)^{\binom{r}{j}}}
    {\displaystyle \prod_{\substack{0 \leq j \leq r \\ j \not\equiv r
          \bmod 2}}T_{\SRIM +{}q}(1-q^jx)^{\binom{r}{j}}} = \\
    \frac{\displaystyle \prod_{\substack{0 \leq j \leq r \\ j \not\equiv r
          \bmod 2}} (1-q^jx)^{2\binom{r}{j}}}
    {\displaystyle \prod_{\substack{0 \leq j \leq r \\ j \not\equiv r
          \bmod 2}} (1-q^{j+1}x)^{\binom{r}{j}}
      \prod_{\displaystyle\substack{0 \leq j \leq r \\ j \equiv r
          \bmod 2}}(1-q^jx)^{\binom{r}{j}}} = 
    \frac{\displaystyle\prod_{\substack{0 \leq j \leq r \\ j \not\equiv r
          \bmod 2}} (1-q^jx)^{2\binom{r}{j}}}
    {\displaystyle\prod_{\substack{0 \leq j \leq r+1 \\ j \equiv r
          \bmod 2}} (1-q^jx)^{\binom{r+1}{j}}}
  \end{multline*}
 by using properties of the $T_{\SRIM{\pm}{}q}$-transform \cite[Chp
   3]{jmm:eulergl+} and Lemma~\ref{lemma:ofir}. When $q$ is even, the
   computations are essentially identical.
\end{proof}

\begin{proof}[Proof of Theorem~\ref{thm:main}]
  The formula of Theorem~\ref{thm:main} is the solution to the
  recurrence \eqref{eq:recur} given the result of
  Lemma~\ref{lemma:solrecur}.
\end{proof}


\begin{proof}[Proof of Corollary~\ref{cor:expmain}]
  The logarithm of the $(r+1)$th generating function $\FSp{r+1}(q,x)$ is
  \begin{equation*}
    \sum_{\substack{0 \leq j \leq r \\ j
        \not\equiv r \bmod 2}} -\binom{r}{j}\log(1-q^jx) =
     \sum_{\substack{0 \leq j \leq r \\ j
        \not\equiv r \bmod 2}} \binom{r}{j} \sum_{n \geq 1}
    \frac{(q^jx)^n}{n}  =
    \sum_{n \geq 1} \sum_{\substack{0 \leq j \leq r \\ j
        \not\equiv r \bmod 2}} 2 \binom{r}{j} q^{nj} \frac{x^n}{2n} \\
    = \sum_{n \geq 1} ((q^n+1)^r-(q^n-1)^r) \frac{x^n}{2n}
  \end{equation*}
  \end{proof}


The binomial formula applied to right hand side of
Theorem~\ref{thm:main} gives the more direct expression
\begin{equation}\label{eq:rchiSpdirect}
  \displaystyle
  -\rchi_{r+1}(\Sp{2n}q) =
  \sum_{\substack{n_0+\cdots+n_r=n \\ j \equiv r \bmod 2 \implies n_j
      =0}}
  \prod_{0 \leq j \leq r} (-1)^{n_j} 
  \binom{-\binom{r}{j}}{n_j}  q^{jn_j} \qquad (r \geq 1)
  \end{equation}
  where the sum ranges over all the
  $\binom{n+\lfloor \frac{1}{2}(r-1)\rfloor}{n}$ weak compositions
  $n_0 + \cdots + n_r$ of $n$ into $r+1$ parts \cite[p 15]{stanley97}
  with $n_j =0$ for all  $j \equiv r \bmod 2$.

Elementary properties of the binomial coefficients imply that the
generating functions satisfy the recurrence $\FSp 1(q,x) = 1$ and 
  \begin{equation*} 
    \FSp{r+1}(q,x) = \frac{\FSp r(q,qx)}{1-q^{r-1}x}   \prod_{1 \leq j  \leq r-1} \FSp{r-j}(q,q^{j-1}x) 
  \end{equation*}
for all $r \geq 1$.

\begin{rmk}\label{rmk:KRconj}
The (non-block-wise form of the) Kn\"orr--Robinson conjecture
\cite{knorr_robinson:89} \cite[\S3]{thevenaz93Alperin}
 for the group
$\Sp{2n}q$ relative to the characteristic $s$ of $\F_q$ states that
\begin{equation*}
  -\rchi_2(\Sp{2n}q) = z_s(\Sp{2n}q)
\end{equation*}
where
$z_s(\Sp{2n}q) = | \{ \chi \in \Irr{\C}{\Sp{2n}q} \mid |\Sp{2n}q|_s
 = \chi(1)_s \}| $ is the number of irreducible complex representations
of $\Sp{2n}q$ of $s$-defect $0$. As
$\FSp 2(q,x)=(1-x)^{-1} = 1+x+x^2+\cdots$, the left side is $1$ and so
is the right side \cite[Remark p 69]{humphreys06}. This confirms the
Kn\"orr--Robinson conjecture for $\Sp{2n}q$ relative to the defining
characteristic.
\end{rmk}

\section{Proof of Theorem~\ref{thm:FSpprimary}}
\label{sec:prim-equiv-reduc}

Let $p$ be a prime and, as in the previous sections, $q$ a prime
power. (The prime $p$ may or may not divide the prime power $q$
although it will soon emerge that $p \nmid q$ is the most interesting
case.)  In this section we discuss Tamanoi's $p$-primary equivariant reduced
\Euc s of the $\Sp{2n}q$-poset $\Li_{2n}^*(\F_q)$ of nonzero totally
isotropic subspaces.

$\Z_p$ denotes the ring of $p$-adic integers and $Z_p^r$ the product
$\Z \times \Z_p^{r-1}$ of one copy of the integers with $r-1$ copies
of the $p$-adic integers.

\begin{defn}\cite[(1-5)]{tamanoi2001}\label{defn:pprimeuc}
  The $r$th, $r \geq 1$, $p$-primary equivariant reduced \Euc\ of the
  $\Sp{2n}q$-poset $\Li_{2n}^*(\F_q)$ is the normalised sum
\begin{equation*}
  \rchi_r(p,\Sp{2n}q) = \frac{1}{|\Sp{2n}q|}
  \sum_{X \in \Hom{Z_p^r}{\Sp{2n}q}}\rchi(C_{\Li_{2n}^*(q)}(X(\Z \times \Z_p^{r-1})))
\end{equation*}
of reduced \Euc s.
\end{defn}

In this definition, the sum ranges over all commuting $r$-tuples
$(X_1,X_2,\ldots,X_r)$ of elements of $\Sp{2n}q$ such that the elements
$X_2,\ldots,X_r$ have $p$-power order. The first $p$-primary
equivariant reduced \Euc\ is independent of $p$ and agrees with the
first equivariant reduced \Euc .

The $r$th $p$-primary equivariant {\em unreduced\/} \Euc\
$\chi_r(p,\Sp{2n}q)$ agrees with the
\Euc\ of the homotopy orbit space $\B\!\Li_{2n}^*(\F_q)_{h\!\Sp{2n}q}$ computed in
Morava $K(r)$-theory at $p$
\cite{HKR2000}
\cite[Remark 7.2]{jmm:partposet2017}
\cite[2-3, 5-1]{tamanoi2001}.

For $r=1$, the $p$-primary equivariant reduced \Euc\ and the
equivariant reduced \Euc\ agree, $\rchi_1(p,\Sp{2n}q) =
\rchi_1(\Sp{2n}q)$, and for $r \geq 1$, as in Lemma~\ref{lemma:recur},
\begin{equation}\label{eq:precur}
  \rchi_{r+1}(p,\Sp{2n}q) = \sum_{[g] \in [\Sp{2n}q_p]}
  \rchi_r(p,C_{\Li_{2n}^*(\F_q)}(g), C_{\Sp{2n}q}(g)) \qquad (n \geq 1)
\end{equation}
where the sum runs over $p$-power order conjugacy classes
in the symplectic group.

The $r$th $p$-primary  generating function at $q$ is the integral power series
\begin{equation}\label{eq:primgenfct}
  \FSp r(p,q,x)  = 1- \sum_{ n \geq 1} \rchi_{r}(p,\Sp{2n}q) x^n \in \Z[[x]]
\end{equation}
associated to the sequence $(-\rchi_r(p,\Sp{2n}q))_{n \geq 1}$ of the
{\em negative\/} of the $p$-primary equivariant reduced \Euc s. We
have $\FSp 1(p,q,x) = \FSp 1(q,x) = 1$ and, directly from the
definition and Lemma~\ref{lemma:contrability}, $\FSp r(p,q,x) = 1$ for
all $r \geq 1$ when $p \mid q$. Thus we now restrict to the case where
$p$ does not divide $q$.

 \begin{defn}\cite[Definition 3.2]{lidlnieder97}\cite[Definition, Chp
  4, \S1]{irelandrosen90}\label{defn:ordf}
  Let $f \in \F_q[x]$ be a \pol\ with $f(0) \neq 0$. The order of $f$,
  $\ord{}f$, is the least positive integer $e$ for which $f(x) \mid
  x^e-1$.

  Let $a$ and $n$ be relatively prime integers. The multiplicative
  order of $a$ modulo $n$, $\ord na$, is the order of $a$ in the unit group
  $(\Z/n\Z)^\times$ of the modulo $n$ residue ring $\Z/n\Z$.
\end{defn}

\begin{defn}\label{defn:Apdq}
  For every integer $n \geq 1$, prime number $p$, and prime power $q$,
  \begin{itemize}
    \item $\IM n{p,q}$ is the number of Irreducible Monic $p$-power
      order \pol s $p(x)$ of degree
    $n$ over $\F_{q}$ with $p(0) \neq 0$
 \item $\SRIM -n{p,q}$  is the number of Self-Reciprocal Irreducible Monic
  $p$-power order \pol s $p(x)$ of even
  degree $2n$ over $\F_{q}$  
\item $\SRIM +d{p,q}$ is the number of unordered pairs
  $\{ p(x), p^*(x)/p(0) \}$ of irreducible monic $p$-power order \pol
  s $p(x)$ of degree $n$ over $\F_{q}$ with $p(0) \neq 0$ and
  $p(x) \neq p^*(x)/p(0)$
  \end{itemize}
\end{defn}

In degree $d=1$,  in particular, $\IM 1{p,q} = (q-1)_p$, represented
by the \pol s $x-\alpha$ with $\alpha$ in the \syl p of the unit group $\F_q^\times$, and
\begin{equation} \label{eq:SRIM1+pq}
  2\SRIM +1{p,q} =
  \begin{cases}
    (q-1)_2-2 & p = 2\\
    (q-1)_p-1 & p > 2
  \end{cases}
\end{equation}
as $x-\alpha$ is fixed if and only if $\alpha^2=1$.
By the $p$-version \cite[(4,7)]{jmm:eulergl+} of a classical identity
\cite[Theorem~3.25]{lidlnieder97}
and by the definition of $\SRIM{\pm}n{p,q}$, 
\begin{equation}\label{eq:pversion}
  \IM n{p,q} = \frac{1}{n} \sum_{d \mid n} \mu(d)(q^{n/d}-1)_p,
  \qquad
  \IM n{p,q} =
  \begin{cases}
    2 \SRIM +1{p,q} + \varepsilon & n=1 \\
    2 \SRIM +n{p,q} & \text{$n>1$ odd} \\
    2 \SRIM +n{p,q} + \SRIM -{n/2}{p,q} & \text{$n>1$ even}
  \end{cases}
\end{equation}
where $\varepsilon=2$ if $p=2$ and $\varepsilon=1$ if $p>2$.

\begin{lemma}\label{lemma:SRIMp}
  Assume $p \nmid q$.  Let $D=\ord p{q^2}$ and let $f \in \F_q[x]$
  be a self-reciprocal irreducible monic $p$-power order \pol\ of
  degree $2d$ for some $d \geq 1$. Then
  \begin{enumerate}
  \item $q^{d} \equiv -1 \bmod p^j$ for some $j \geq 1$
  \item $D \mid d$
  \item $f(x) \mid (x^{(q^{2d}-1)_p}-1)$ and $f(x) \mid (x^{q^d+1}-1)$
  \item $f(x) \mid (x^{(q^{d}+1)_p}-1)$ 
  \end{enumerate}
\end{lemma}
\begin{proof}
  Let $f \in \F_q[x]$ be a self-reciprocal irreducible monic $p$-power
  order \pol\ of degree $2d$, $d \geq 1$.  Then
  $p \mid \ord{}f \mid q^{2d}-1$ by \cite[Corollary
  3.4]{lidlnieder97}. In other words, $q^{2d} \equiv 1 \bmod \ord{}f$,
  $q^{2d} \equiv 1 \bmod p$, and thus $d$ is a multiple of
  $D$. Moreover, $f(x) \mid (x^{(q^{2d}-1)_p}-1)$ by \cite[Lemma
  3.6]{lidlnieder97} as $\ord{}f \mid (q^{2d}-1)_p$, and
  $f(x) \mid (x^{q^d+1}-1)$ by \cite[Theorem 1.(i)]{meyn1990}. But
  then $f(x) \mid (x^{(q^{d}+1)_p}-1)$ by
  \cite[Corollary~3.7]{lidlnieder97} as $f(x)$ is irreducible and
  $\GCD{q^{2d}-1)_p,q^d+1} = \GCD{(q^{2d}-1)_p,(q^d+1)_p}
  =\GCD{(q^{d}-1)_p(q^d+1)_p,(q^d+1)_p} =(q^d+1)_p$.

  The irreducible factors of $x^{(q^d+1)_p}-1$ of degree $\geq 2$ are
  the irreducible factors of
  the cyclotomic \pol s $\Phi_{p^j}(x)$ where $j \geq 1$ and $p^j \mid
  q^d+1$. Thus $q^d \equiv -1 \bmod p^j$ for some $j \geq 1$.
\end{proof}


\begin{lemma}\label{lemma:SRIMmpfactor}
  Assume $p \nmid q$ and $d \geq 1$.  Each irreducible monic factor
  $f \in \F_q[x]$ of $x^{(q^n+1)_p}-1$, $n \geq 1$, of degree $2d \geq 2$
  is self-reciprocal, has $p$-power order, and $d \mid n$ with $n/d$
  odd where $\deg(f) = 2d$.
\end{lemma}
\begin{proof}
  Suppose $f(x) \mid (x^{(q^{n}+1)_p}-1)$. Then $f(0) \neq 0$ and
  $f(x)$ has $p$-power order by \cite[Lemma 3.6]{lidlnieder97}. Since
  $f$ is irreducible of degree $\deg(f) \geq 2$ and
  $f(x) \mid x^{(q^n+1)_p}-1 \mid x^{q^n+1}-1$, $f(x)$ is
  self-reciprocal and $d=\frac{1}{2}\deg(f)$ divides $n$ with odd
  quotient $n/d$ by \cite[Theorem 1.(ii)]{meyn1990}.
\end{proof}


\begin{lemma}\label{lemma:SRIMmpq}
   Assume $p \nmid q$. For any $n \geq 1$,
  \begin{equation*}
    2n \SRIM -n{p,q} =
    \begin{cases}
      \sum\limits_{d \mid \frac{n}{n_2}} \mu(d)(q^{n/d}+1)_p & n \neq n_2 \\
      (q^n+1)_p-\varepsilon & n = n_2
    \end{cases}
  \end{equation*}
  where $\varepsilon=2$ if $p=2$ and $\varepsilon=1$ if $p>2$. For any
  {\em odd\/} $n \geq 1$,
  $\SRIM -{2^kn}{p,q} = 2^{-k}\SRIM -{n}{p,q^{2^k}}$ for all $k \geq 0$.
\end{lemma}
\begin{proof}
  Recall that the irreducible factors of the \pol\ $x^{q^n+1}-1$ are distinct and that there is
   one linear factor, $x+1$, of order $1$, when  $q$ is even and
   two, $x+1$, $x-1$, of order $1$ and $2$, if $q$ is odd.
   The \pol\ $x^{(q^n+1)_p}-1$ thus has $\varepsilon$ linear factors of $p$-power order
   where $\varepsilon=2$ if $q$ is odd and $p=2$ and $\varepsilon=1$ in all other cases.
  Lemma~\ref{lemma:SRIMmpfactor} and \Mb\ inversion thus imply that
  \begin{equation*}
    (q^n+1)_p =  
    \varepsilon + \sum_{\substack{d \mid n \\ \text{$n/d$ odd}} } 2d \SRIM -d{p,q}, \qquad
    2n \SRIM -n{p,q} = \sum_{\substack{d \mid n \\ \text{$d$ odd}} } \mu(d) ((q^{n/d}+1)_p - \varepsilon)
    \end{equation*}
    The first part of the lemma follows because $\sum_{d \mid
      \frac{n}{n_2}} \mu(d)$ is $1$ if $n=n_2$ is a power of $2$ and
    $0$ otherwise. If $n=2^km$, $k \geq 0$, $m \geq 1$ odd, then $2n \SRIM -n{p,q} = \sum_{d \mid m} \mu(d)(q^{2^km/d}+1)_p = 2m \SRIM -m{p,q^{2^k}}$.
\end{proof}

\begin{lemma} \label{lemma:SRIM+pq}
  Assume $p \nmid q$. For every odd $n \geq 1$ and $k \geq 0$,
  \begin{align*}
    &\IM {2^kn}{p,q} = 2^{-k} \IM n{p,q^{2^k}} - 2^{-k} \IM n{p,q^{2^{k-1}}} \\
     &\SRIM +{2^kn}{p,q} = 2^{-k} \SRIM +n{p,q^{2^k}} - 2^{-k}
     \SRIM +n{p,q^{2^{k-1}}} -2^{-k} \SRIM -n{p,q^{2^{k-1}}}
  \end{align*}
 \end{lemma}
 \begin{proof}
  Since $\mu(2^jn)=0$ for $j \geq 2$,
  \begin{align*}
    n2^k\IM {n2^k}{p,q} &= \sum_{d \mid n} \mu(n/d)(q^{2^kd}-1)_p +
    \sum_{d \mid n} \mu(2n/d)(q^{2^{k-1}d}-1)_p \\
    &= \sum_{d \mid n} \mu(n/d)(q^{2^kd}-1)_p - \sum_{d \mid n}
      \mu(2n/d)(q^{2^{k-1}d}-1)_p
      = n \IM n{p,q^{2^k}} - n \IM n{p,q^{2^{k-1}}}
  \end{align*}
for all $k \geq 0$. This proves the first assertion. Now,
\begin{align*}
  \SRIM +{2^kn}{p,q} &\stackrel{\eqref{eq:pversion}}{=}
                       2^{-1} \IM{2^kn}{p,q} - 2^{-1} \SRIM -{2^{k-1}n}{p,q}    \\& =
  2^{-k} (2^{-1}\IM n{p,q^{2^k}} - 2^{-1} \IM n{p,q^{2^{k-1}}}) -
                                                                                    2^{-1}
                                                                                    \SRIM
                                                                                    -{2^{k-1}n}{p,q} \\
                     &\stackrel{\eqref{eq:pversion}, L.~\ref{lemma:SRIMmpq}}{=}
  2^{-k} \SRIM +n{p,q^{2^k}} - 2^{-k} \SRIM +n{p,q^{2^{k-1}}}-
                                                                                    2^{-k} \SRIM -{n}{p,q^{2^{k-1}}}                     
\end{align*}
proves the second assertion.
\end{proof}

For all pairs $(p,q)$, where $p$ is a prime, $q$ is a prime power,
  and $p \nmid q$, and for all $r \geq 1$, the $p$-primary analogue of \eqref{eq:recur}, 
\begin{equation}
  \label{eq:FSppqrecur}
  \FSp{r+1}(p,q)(x) =
  \begin{cases}
  \FSp {r}(p,q)(x)^2
  T_{\SRIM -{}{p,q}}(\FGL{-}r(p,q,x))
  T_{\SRIM +{}{p,q}}(\FGL{+}r(p,q,x)) & p=2 \\
  \FSp {r}(p,q)(x)
  T_{\SRIM -{}{p,q}}(\FGL{-}r(p,q,x))
  T_{\SRIM +{}{p,q}}(\FGL{+}r(p,q,x)) & p > 2
  \end{cases} 
\end{equation}
is a consequence of recurrence \eqref{eq:precur}. Note here that
a semisimple $g
\in \Sp {2n}q$ has $p$-power order if and only if all the
irreducible factors in its characteristic \pol, described in
Proposition~\ref{prop:srfact}, have $p$-power order.
This is because multiplication by $x$ in the $\F_q[x]$-module
$\F_q[x]/(r(x))$, where $r(x)$ is irreducible with $r(0) \neq 0$, has
$p$-power order if and only if $r(x)$ has $p$-power order by
\cite[Lemma~3.5]{lidlnieder97}.
Also note that in Proposition~\ref{prop:srfact} with odd $q$, the
\pol\ $(x+1)$ has even order and is therefore allowed only when $p=2$.


The proof of Theorem~\ref{thm:FSpprimary} consists in verifying
that the solution to recurrence~\eqref{eq:FSppqrecur} satisfies the
infinite product expansion
\begin{equation*}
  \FSp r(p,q,x) = \prod_{n \geq 1}(1-x^n)^{c_r(p,q,n)} \qquad
    c_r(p,q,n) = \frac{1}{2n} \sum_{d \mid n}\mu(n/d)((q^d-1)_p^{r-1} - (q^d+1)_p^{r-1})
\end{equation*}
Since the infinite product expansions of $\FGL{\pm}r(p,q,x)$
are \cite[\S1]{jmm:eulergl+} \cite[\S1]{jmm:eulergl-}
  \begin{alignat*}{3}
    &\FGL +r(p,q,x) = \prod_{n \geq 1}(1-x^n)^{a_r^+(p,q,n)} &&\qquad
    &&a_r^+(p,q,n) = \frac{1}{n} \sum_{d \mid n}\mu(n/d)(q^d-1)_p^{r-1} \\
    &\FGL -r(p,q,x) = \prod_{n \geq 1}(1-x^n)^{a_r^-(p,q,n)}&&\qquad
    &&a_r^-(p,q,n) = \frac{1}{n} \sum_{d \mid n} (-1)^d \mu(n/d)(q^d-(-1)^d)_p^{r-1} 
  \end{alignat*}
  and \cite[(3.2)]{jmm:eulergl+}
  \begin{equation*}
    T_{\SRIM{\pm}{}{p,q}}\FGL{\pm}r(p,q,x) = \prod_{d \geq 1} \FGL{\pm}r(p,q^d,x^d)^{\SRIM{\pm}d{p,q}} = 
    \prod_{n,d \geq 1}
    (1-x^{nd})^{a_r^{\pm}(p,q^d,n) \SRIM{\pm}d{p,q}} 
  \end{equation*}
we must show that
  \begin{equation}\label{eq:crecur}
    c_{r+1}(p,q,N) =
    \begin{cases}
      2c_{r}(p,q,N) + \sum\limits_{d \mid N} a_r^-(p,q^d,N/d) \SRIM -d{p,q}  +
      \sum\limits_{d \mid N} a_r^+(p,q^d,N/d) \SRIM +d{p,q}
      & p=2 \\
      c_{r}(p,q,N) + \sum\limits_{d \mid N} a_r^-(p,q^d,N/d) \SRIM -d{p,q}  +
      \sum\limits_{d \mid N} a_r^+(p,q^d,N/d) \SRIM +d{p,q} & p>2
    \end{cases}
  \end{equation}
  for all $N \geq 1$ and all $r \geq 0$.

Theorem~\ref{thm:FSpprimary} will here be proved only for odd primes
$p$. The below proof can easily be modified to cover 
$p=2$.
  
 \begin{proof}[Proof of Theorem~\ref{thm:FSpprimary} for $p>2$]
   Assume $p>2$ and let $N \geq 1$ be an odd integer. Induction shows that
   \begin{align}
                \label{eq:aruniversal}                                  
                 &2^ka_r^{\pm}(p,q,2^kN) =
              \begin{cases}
                 a_r^+(p,q^2,N) - \frac{1}{2}
                a_r^{\pm}(p,q,N) & k=1 \\
                 a_r^+(p,q^{2^k},N) - 
                a_r^+(p,q^{2^{k-1}},N) & k>1
              \end{cases}
            \shortintertext{and}                                                                                
      \label{eq:cruniversal}
     &2^kc_r(p,q,2^kN) =
     \begin{cases}
       a_r^-(p,q,N) + a_r^+(p,q,N) & k=0 \\
       c_r(p,q^{2^k},N) -  c_r(p,q^{2^{k-1}},N) & k>0
     \end{cases}                                  
   \end{align}
Indeed,  $a_r^{\pm}(p,q,2N) = \frac{1}{2} a_r^+(p,q^2,N) -\frac{1}{2} a_r^{\pm}(p,q,N)$
    because
    \begin{align*}
      &
        \begin{aligned}[t]
         2N a_r^-(p,q,2N) &= \sum_{d \mid N} -\mu(2N/d)(q^d+1)^{r-1}_p +
      \sum_{d \mid N} \mu(N/d)(q^{2d}-1)^{r-1}_p \\ &=
      -\sum_{d \mid N} -\mu(N/d)(q^d+1)^{r-1}_p + \sum_{d \mid N} \mu(N/d)(q^{2d}-1)^{r-1}_p=
      -N a_r^-(p,q,N) + Na_r^+(p,q^2,N)          
        \end{aligned}
\\
      &2N a_r^+(p,q,2N) = \sum_{d \mid N} \mu(2N/d)(q^d-1)^{r-1}_p +
      \sum_{d \mid N} \mu(N/d)(q^{2d}-1)^{r-1}_p =
      -N a_r^+(p,q,N) + N a_r^+(p,q^2,N)
    \end{align*}
Since $\mu(2^jN)=0$ for $j \geq 2$,
 \begin{multline*}
   2^kNa_r^{\pm}(p,q,2^kN) =
  \sum_{d \mid N} \mu(N/d)(q^{2^kd}-1)_p^{r-1}  +
  \sum_{d \mid N} \mu(2N/d)(q^{2^{k-1}d}-1)_p^{r-1}
  = Na_r^+(p,q^{2^k},N) - Na_r^+(p,q^{2^{k-1}},N)
\end{multline*}
for all $k > 1$.
   Since $N$ is odd, 
    $c_r(p,q,N) =
    \frac{1}{2}a_r^-(p,q,N) + \frac{1}{2}a_r^+(p,q,N)$, and, as
    $\mu(2^jN)=0$ for $j \geq 2$, 
   $ 2^kNc_r(p,q^{2^k},N) = N c_r(p,q^{2^k},N) -N
   c_r(p,q^{2^{k-1}},N)$ when $k>0$.

   The first equality of the  below display holds (at $d=1$) because   
   $\SRIM -1{p,q} + \frac{1}{2} = \frac{1}{2}(q+1)_p$ by
   Lemma~\ref{lemma:SRIMmpq}. The next to last equality holds because
   \begin{equation*}
  \sum_{\{d_1 \colon f \mid d_1 \mid d_2\}} \mu(d_1/f) =
  \begin{cases}
    1 & f=d_2 \\ 0 & f < d_2 
  \end{cases}
\end{equation*}
contributes only when $f=d_2$. Remembering these observations we find that
    \begin{multline}\label{eq:1d1Nminus}
      \sum_{d \mid N} a_r^-(p,q^d,N/d) \SRIM -d{p,q} + \frac{1}{2}a_r^-(p,q,N) =
      \sum_{d \mid N} a_r^-(p,q^d,N/d)\frac{1}{2d} \sum_{f \mid d}
      \mu(f)(q^{d/f}+1)_p \\
      \begin{split}
      &=
      \sum_{d \mid N} a_r^-(p,q^d,N/d) \frac{1}{2d} \sum_{f \mid d} \mu(d/f)(q^f+1)_p =
      -\sum_{d \mid N} \frac{d}{N}\sum_{e \mid \frac{N}{d}}
      \mu(N/de)(q^{de}+1)^{r-1}_p \frac{1}{2d} \sum_{f \mid d}
      \mu(d/f)(q^f+1)_p \\ &=
      -\frac{1}{2N} \sum_{f \mid d_1 \mid d_2 \mid N} \mu(N/d_2) (q^{d_2}+1)_p^{r-1} (q^f+1)_p \mu(d_1/f) =
      -\frac{1}{2N} \sum_{d \mid N} \mu(N/d) (q^{d}+1)_p^{r} =
      \frac{1}{2} a_{r+1}^-(p,q,N)   
      \end{split}
    \end{multline}
    The first equality of the  below display holds (at $d=1$) because   
   $\SRIM +1{p,q} + \frac{1}{2} = \frac{1}{2}(q-1)_p =
    \frac{1}{2}\IM 1{p,q}$ by \eqref{eq:SRIM1+pq} and
    \eqref{eq:pversion}. For all odd $d>1$, $\SRIM +d{p,q} =
    \frac{1}{2} \IM d{p,q}$ by \eqref{eq:pversion}. Remembering these observations we find that
    \begin{multline}\label{eq:1d1Nplus}
      \sum_{d \mid N} a_r^+(p,q^d,N/d) \SRIM +d{p,q} + \frac{1}{2}a_r^+(p,q,N) =
      \sum_{d \mid N} a_r^+(p,q^d,N/d) \frac{1}{2}\IM d{p,q} \\
      \begin{split}
      &=
      \sum_{d \mid N} a_r^+(p,q^d,N/d) \frac{1}{2d} \sum_{f \mid d} \mu(f)(q^{d/f}-1)_p =
      \sum_{d \mid N} a_r^+(p,q^d,N/d) \frac{1}{2d} \sum_{f \mid d} \mu(d/f)(q^{f}-1)_p \\&=
      \sum_{d \mid N} \frac{d}{N}
      \sum_{e \mid N/d} \mu(N/de)(q^{de}-1)_p^{r-1}\frac{1}{2d} \sum_{f \mid d} \mu(d/f)(q^{f}-1)_p \\&=
      \frac{1}{2N} \sum_{f \mid d_1 \mid d_2 \mid N} \mu(N/d_2)(q^{d_2}-1)_p^{r-1}(q^{f}-1)_p \mu(d_1/f) =
      \frac{1}{2N} \sum_{d \mid N} \mu(N/d)(q^d-1)_p^r =
      \frac{1}{2} a_{r+1}^+(p,q,N)  
      \end{split}
    \end{multline}
    By adding \eqref{eq:1d1Nminus} and \eqref{eq:1d1Nplus} we get
    \begin{multline}
      \label{eq:proofN1}
      \sum_{d \mid N} a_r^-(p,q^d,N/d) \SRIM -d{p,q} +
      \sum_{d \mid N} a_r^+(p,q^d,N/d) \SRIM +d{p,q}
      \\ =
      \frac{1}{2}a_{r+1}^-(p,q,N) - \frac{1}{2}a_{r}^-(p,q,N)
      +\frac{1}{2}a_{r+1}^+(p,q,N) - \frac{1}{2}a_{r}^+(p,q,N)  
    \stackrel{\eqref{eq:cruniversal}}{=} c_{r+1}(p,q,N) -  c_r(p,q,N)  
    \end{multline}
    proving \eqref{eq:crecur}  for all odd $N$.

 Next consider $2N$, $N$ odd.  
 The expression $\sum_{d \mid 2N} a_r^-(p,q^d,2N/d) \SRIM -d{p,q} +
 \sum_{d \mid 2N} a_r^-(p,q^d,2N/d) \SRIM +d{p,q}$ is the sum of
 the four terms
    \begin{align*}
      \sum_{d \mid N} a_r^-(p,q^d,2N/d) \SRIM -d{p,q}
      &\stackrel{\eqref{eq:aruniversal}}{=}
        \frac{1}{2} \sum_{d \mid N} (a_r^+(p,q^{2d},N/d)-a_r^-(p,q^d,N/d)) \SRIM -d{p,q} \\
        \sum_{d \mid N} a_r^-(p,q^{2d},N/d) \SRIM -{2d}{p,q}
      &\stackrel{\text{L.~\ref{lemma:SRIMmpq}}}{=}
        \frac{1}{2} \sum_{d \mid N} a_r^-(p,q^{2d},N/d) \SRIM
        -d{p,q^2} \\
         \sum_{d \mid N} a_r^+(p,q^d,2N/d) \SRIM +d{p,q}
      &\stackrel{\eqref{eq:aruniversal}}{=}
         \frac{1}{2} \sum_{d \mid N} (a_r^+(p,q^{2d},N/d)-a_r^+(p,q^d,N/d)) \SRIM +d{p,q} \\
      \sum_{d \mid N} a_r^+(p,q^{2d},N/d) \SRIM +{2d}{p,q}
      &\stackrel{\text{L.~\ref{lemma:SRIM+pq}}}{=}
        \frac{1}{2} \sum_{d \mid N} a_r^+(p,q^{2d},N/d)
        (\SRIM +d{p,q^2} - \SRIM +d{p,q} - \SRIM -d{p,q})
    \end{align*}
 which is
 \begin{multline*}
    2^{-1}(\sum_{d \mid N} a_r^-(p,q^{2d},N/d) \SRIM
    -d{p,q^2}
    +\sum_{d \mid N} a_r^+(p,q^{2d},N/d) \SRIM
    +d{p,q^2}) \\
   -2^{-1}(\sum_{d \mid N} a_r^-(p,q^d,N/d) \SRIM -d{p,q}
   +\sum_{d \mid N} a_r^+(p,q^d,N/d) \SRIM +d{p,q}) \\
   \begin{split}
     &\stackrel{\eqref{eq:proofN1}}{=}
    -2^{-1}(c_{r+1}(p,q,N) - c_r(p,q,N))
    +2^{-1}(c_{r+1}(p,q^2,N) - c_r(p,q^2,N)) \\
    &\, \, \, \, =
2^{-1}(c_{r+1}(p,q^2,N)
    -c_{r+1}(p,q,N))
    -2^{-1}
    (c_{r}(p,q^2,N) - c_{r}(p,q,N)) \\
    &\stackrel{\eqref{eq:cruniversal}}{=}
    c_{r+1}(p,q,2N) -c_{r}(p,q,2N)
    \end{split}
  \end{multline*}
 This proves \eqref{eq:crecur} for $2N$,  $N$ odd.

Finally, we consider $2^kN$, $N$ odd, $k > 1$. 
We shall evaluate the sum
\begin{multline*}
   \sum_{d \mid 2^kN} a_r^-(p,q^d,2^kN/d) \SRIM -d{p,q} + 
    \sum_{d \mid 2^kN}  a_r^+(p,q^d,2^kN/d) \SRIM +d{p,q}
   \\= \sum_{0 \leq j \leq k} \sum_{d \mid N}
   a_r^-(p,q^{2^jd},2^{k-j}N/d) \SRIM -{2^jd}{p,q} +
   \sum_{0 \leq j \leq k} \sum_{d \mid N}  a_r^+(p,q^{2^jd},2^{k-j}N/d) \SRIM +{2^jd}{p,q}
\end{multline*}
which occurs on the right hand side of \eqref{eq:crecur}. For $j=0$ we get
\begin{equation*}
  \sum_{d \mid N} a_r^{\pm}(p,q^d,2^kN/d) \SRIM {\pm}d{p,q}
  \stackrel{\eqref{eq:aruniversal}}{=}
  2^{-k}\big(\sum_{d \mid N}
  (a_r^{+}(p,q^{2^kd},N/d)-a_r^{+}(p,q^{2^{k-1}d},N/d)) \SRIM
  {\pm}d{p,q} \big)
\end{equation*}
For $0 < j < k$ we get
\begin{align*}
&\begin{aligned}
     \sum_{d \mid N} a_r^{-}(p,q^{2^jd},2^{k-j}N/d) \SRIM {-}{2^jd}{p,q}
  &\stackrel{\text{L.~\ref{lemma:SRIMmpq}}}{=}
  2^{-j}\sum_{d \mid N} a_r^{-}(p,q^{2^jd},2^{k-j}N/d) \SRIM
  {-}{d}{p,q^{2^j}} \\
  &\stackrel{\eqref{eq:aruniversal}}{=}
  2^{-k}\sum_{d \mid N} (a_r^{+}(p,q^{2^kd},N/d)-a_r^{-}(p,q^{2^{k-1}d},N/d)) \SRIM
  {-}{d}{p,q^{2^j}} 
    \end{aligned}
\\
  &\begin{aligned}
    \sum_{d \mid N} a_r^{+}(p,&q^{2^jd},2^{k-j}N/d) \SRIM {+}{2^jd}{p,q}
  \\ & \stackrel{\text{L.~\ref{lemma:SRIM+pq}}}{=}
  2^{-j}\sum_{d \mid N} a_r^{+}(p,q^{2^jd},2^{k-j}N/d) (\SRIM
  {+}{d}{p,q^{2^j}} - \SRIM {+}{d}{p,q^{2^{j-1}}} - \SRIM {-}{d}{p,q^{2^{j-1}}})\\
  &\stackrel{\eqref{eq:aruniversal}}{=}
  2^{-k}\sum_{d \mid N}
  (a_r^{+}(p,q^{2^kd},N/d)-a_r^{+}(p,q^{2^{k-1}d},N/d))
  (\SRIM
  {+}{d}{p,q^{2^j}} - \SRIM {+}{d}{p,q^{2^{j-1}}} - \SRIM {-}{d}{p,q^{2^{j-1}}})
    \end{aligned}
\end{align*}
For $j=k$ we get
\begin{align*}
  &\sum_{d \mid N} a_r^{-}(p,q^{2^kd},N/d) \SRIM {-}{2^kd}{p,q}
  \stackrel{\text{L.~\ref{lemma:SRIMmpq}}}{=}
  2^{-k}\sum_{d \mid N} a_r^{-}(p,q^{2^kd},N/d) \SRIM
  {-}{d}{p,q^{2^k}} \\
  &\sum_{d \mid N} a_r^{+}(p,q^{2^kd},N/d) \SRIM {+}{2^kd}{p,q}
  \stackrel{\text{L.~\ref{lemma:SRIM+pq}}}{=}
  2^{-k}\sum_{d \mid N} a_r^{+}(p,q^{2^kd},N/d)
  (\SRIM {+}{d}{p,q^{2^k}} - \SRIM +d{p,q^{2^{k-1}}} - \SRIM -d{p,q^{2^{k-1}}}) 
\end{align*}
The sum of these $2(k+1)$ terms is
\begin{multline*}
  2^{-k}\sum_{d \mid N} a_r^{-}(p,q^{2^kd},N/d) \SRIM -d{p,q^{2^k}}
  +
  2^{-k}\sum_{d \mid N} a_r^{+}(p,q^{2^kd},N/d) \SRIM +d{p,q^{2^k}} \\
  -2^{-k}(\sum_{d \mid N} a_r^{-}(p,q^{2^{k-1}d},N/d) \SRIM -d{p,q^{2^{k-1}}}
  +
  \sum_{d \mid N} a_r^{+}(p,q^{2^{k-1}d},N/d) \SRIM +d{p,q^{2^{k-1}}})
  \\ 
  \begin{split}
    &\stackrel{\eqref{eq:proofN1}}{=}
    2^{-k}( c_{r+1}(p,q^{2^k},N) -  c_{r}(p,q^{2^k},N)) -
  2^{-k}( c_{r+1}(p, q^{2^{k-1}},N) - c_{r}(p, q^{2^{k-1}},N)) \\
  &\phantom{2}=2^{-k}( c_{r+1}(p,q^{2^k},N)- c_{r+1}(p,q^{2^{k-1}},N)) -
  2^{-k}( c_{r}(p,q^{2^k},N) - c_{r}(p,q^{2^{k-1}},N))
  \stackrel{\eqref{eq:cruniversal}}{=}
  c_{r+1}(p,q,2^kN) - c_{r}(p,q,2^kN)
  \end{split}
\end{multline*}
This proves \eqref{eq:crecur} for $2^kN$,  $N$ odd , $k > 1$.

We can now conclude that \eqref{eq:crecur} holds for all $N$ when the
prime $p$ is odd.
\end{proof}

Two prime powers prime to $p$
are declared to be $p$-equivalent if the 
generate the same closed subgroup
of the topological unit group $\Z_p^\times$ in $\Z_p$. More concretely, if we let
  \begin{equation*}
    O(p,q) =
    \begin{cases}
      (q \bmod 8, \nu_2(q^2-1)) & p=2 \\
      (\ord pq, \nu_p(q^{\ord pq}-1)) & p>2 
    \end{cases}
  \end{equation*}
  then the prime powers $q_1$ and $q_2$ are $p$-equivalent,
  $\overline{\gen{q_1}} = \overline{\gen{q_2}} \leq \Z_p^\times$, if and
  only if $O(p,q_1) = O(p,q_2)$ \cite[\S3]{bmo2}.
  The sequences $(\SRIM{\pm}d{p,q})_{d \geq1}$ and hence the power
  series $\FSp r(p,q,x)$ depend only on the $p$-class of $q$ when
  $p \nmid q$.

  For example,
  the $2$-classes are represented by the $2$-adic units $\pm 3^{2^e}$,
  $e \geq 0$, and the $3$-classes by the prime powers $2^{3^e}$ and
  $4^{3^e}$, $e \geq 0$ \cite[Lemma 1.11.(a)]{bmo2}.

  \begin{exmp}\label{exmp:FSpprim}
    For all $r \geq 1$ and $e \geq 0$
    \begin{align*}
      \FSp{r+1}(2,3^{2^e},x) &=
      \begin{cases}
        \displaystyle
       \frac{\prod_{n \geq 0}
          Q(x^{2^{n+1}})^{2^{(n+3)(r-1)}}}{Q(x)^{2^{2(r-1)}}(1+x)^{2^{r-1}}}
        & e=0 \\
        \displaystyle \frac{\prod_{n \geq 0} Q(x^{2^n})^{2^{(n+2)(r-1)+re}}}{(1-x)^{2^{r-1}}} & e>0 
      \end{cases}  \\
        \FSp{r+1}(2,-3^{2^e},x) &=
      \begin{cases}
        \displaystyle
           \frac{\prod_{n \geq -1}
             Q(x^{2^{n+1}})^{2^{(n+3)(r-1)}}}{(1-x)^{2^{r-1}}} & e=0\\
           \displaystyle
           \frac{\prod_{n \geq 0}
             Q(x^{2^{n+1}})^{2^{(n+3)(r-1)+re}}}{Q(x)^{2(r-1)+re}(1+x)^{2^{r-1}}} & e>0
      \end{cases}
    \end{align*}
    where $Q(x) = \frac{1-x}{1+x}$.  This follows from
    Theorem~\ref{thm:FSpprimary} after some power series
    manipulations as
    $\exp(-\sum_{n \geq 1} (2n)_2^r \frac{x^{n}}{n}) = \prod_{n \geq 0}
        Q(x^{2^n})^{2^{(n+1)(r-1)}}$, $\exp(-2 \sum_{n\geq 0} \frac{x^{2n+1}}{2n+1})=Q(x)$, and (in the case of $+3^{2^e}$)
    \begin{equation*}
       (3^n+1)_2 =
        \begin{cases}
          4 & 2 \nmid n \\ 2 & 2 \mid n
        \end{cases} \qquad
     (3^n-1)_2 =
      \begin{cases}
        2 & 2 \nmid n \\ 4n_2 & 2 \mid n
      \end{cases} \qquad
      ((3^{2^e})^n-1)_2 =
      \begin{cases}
        2^{2+e} & 2 \nmid n \\ 2^{2+e}n_2 & 2 \mid n
      \end{cases} \qquad
       ((3^{2^e})^n+1)_2 = 2
    \end{equation*}
    for all $e>0$.
  \end{exmp}

  \section{Proof of Theorem~\ref{thm:WSp}}
\label{sec:other-pres-equiv}

Recall that a (finite) multiset $\lambda$ is a (finite) base set
$B(\lambda)$ with a multiplicity function assigning an integer
$E(\lambda,b) \geq 0$ to every $b \in B(\lambda)$. The number
$|\lambda| = \sum_{b \in B(\lambda)} E(\lambda,b)$ is the cardinality
of the multiset $\lambda$. If the base set consists of natural numbers
$\geq 0$ with$\sum_{b \in B(\lambda)} bE(\lambda,b)=n$, $\lambda$ is
a {\em partition of $n$}, in symbols $\lambda \vdash n$. The multiset sum
$\lambda_1 + \lambda_2$ is the multiset with multiplicity function
$E(\lambda_1+\lambda_2,b) = E(\lambda_1,b) + E(\lambda_2,b)$. A {\em
  partition of $n$ into parts of two kinds} is a pair
  $(\lambda_-, \lambda_+)$ of multisets, $\lambda^-$ and $\lambda^+$,
  such that the multiset sum $\lambda_-+\lambda_+$ partitions $n$, in
  symbols $(\lambda_-, \lambda_+) \vdash n$.
  
\begin{lemma}\label{lemma:Bnq}
  Let $A(q) \in \Q[q]$ be a rational \pol\ in the indeterminate
  $q$. The \pol\ sequence $(B_n(q))_{n \geq 0}$ with $B_0(q)=1$ and
  \begin{equation*}
    B_n(q) = \sum_{\lambda \vdash n} (-1)^{|\lambda|} \prod_{d
      \in B(\lambda)} \frac{A(q^d)^{E(\lambda,d)}}{|C_d \wr
      \Sigma_{E(\lambda,d)}|}, \qquad n \geq 1,
  \end{equation*}
satisfies the recurrence $B_0(q)=1$ and
  $nB_n(q) + \sum_{1 \leq j \leq n} A(q^j) B_{n-j}(q) = 0$ for $n \geq
  1$.
\end{lemma}
\begin{proof}
  The claim is that
  \begin{equation*}
    n \sum_{\lambda \vdash n} (-1)^{|\lambda|} \prod_{d
      \in B(\lambda)} \frac{A(q^d)^{E(\lambda,d)}}{|C_d \wr
      \Sigma_{E(\lambda,d)}|} +
    \sum_{1 \leq j \leq n} A(q^j)\sum_{\mu \vdash n-j} (-1)^{|\mu|}
    \prod_{d \in B(\mu)} \frac{A(q^d)^{E(\lambda,d)}}{|C_d \wr
      \Sigma_{E(\lambda,d)}|} =0 
  \end{equation*}
  for all $n \geq 1$. Since, for all $d \in B(\lambda)$,
  \begin{equation*}
    A(q^d) \prod_{f \in B(\lambda -\{d\})}
    \frac{A(q^f)^{E(\lambda-\{d\})}}{|C_f \wr \Sigma_{E(\lambda
        -\{d\}}|} = d E(\lambda,d)
    \prod_{d
      \in B(\lambda)} \frac{A(q^d)^{E(\lambda,d)}}{|C_d \wr
      \Sigma_{E(\lambda,d)}|}
  \end{equation*}
  it suffices to show that
\begin{equation*}
  (-1)^{|\lambda|} n+ 
  \sum_{d \in B(\lambda)} (-1)^{|\lambda - \{d\}|}
  d E(\lambda,d) = 0
\end{equation*}
But this is obvious since $|\lambda - \{d\}| =
|\lambda|-1$ and
$\sum_{d \in B(\lambda)} d E(\lambda,d) = n$ as $\lambda$ partitions $n$.
\end{proof}

\begin{cor}\label{cor:Bnq} 
  For all $n \geq 1$ and $r \geq 0$,
  \begin{align*}
      -\rchi_{r+1}(\Sp{2n}{q}) &=
    \sum_{\lambda \vdash n} (-1)^{|\lambda|}
    \prod_{d \in B(\lambda)}
                                 \frac{((q^d-1)^r/2-(q^d+1)^r/2)^{E(\lambda,d)}}{|C_d
                                 \wr \Sigma_{E(\lambda,d)}|} \\
-\rchi_{r+1}(p,\Sp{2n}{q}) &=
    \sum_{\lambda \vdash n} (-1)^{|\lambda|}
    \prod_{d \in B(\lambda)} \frac{((q^d-1)_p^r/2-(q^d+1)^r_p/2)^{E(\lambda,d)}}{|C_d \wr \Sigma_{E(\lambda,d)}|}
  \end{align*}
\end{cor}
\begin{proof}
  Let $H(r,q)=\frac{1}{2}(q-1)^r -
  \frac{1}{2}(q+1)^r$ ($H(r,q)=\frac{1}{2}(q-1)_p^r -
  \frac{1}{2}(q+1)_p^r$ in the $p$-primary case). By Theorem~\ref{cor:expmain},  
 $\sum_{n \geq 0} -\rchi_{r+1}(\Sp{2n}q)x^n = \exp(-\sum_{n \geq 1} H(r,q^n)
  \frac{x^n}{n})$ (with the convention that $-\rchi_{r+1}(\Sp{0}q)
  = 1$ for all $r \geq 0$), so
  the sequence $-\rchi_{r+1}(\Sp nq)$,
  $n \geq 1$,  satisfies the recurrence
  \begin{equation*}
    n(-\rchi_{r+1}(\Sp{2n}q) + \sum_{1 \leq j \leq n}
    H(r,q^j)(-\rchi_{r+1}(\Sp{2(n-j)}q) = 0, \qquad n \geq 1,
  \end{equation*}
  according to \cite[Lemma~3.7]{jmm:eulergl+}.
  We can now apply
  Lemma~\ref{lemma:Bnq}. 
  \end{proof}


Let $F_q$ denote the standard Frobenius endo\m\ of the symplectic
algebraic group $\operatorname{Sp}_{2n}(\overline \F_s)$,
$s=\operatorname{char}(\F_q)$, with fixed points
$\operatorname{Sp}_{2n}(\overline \F_s)^{F_q} = \Sp{2n}q$.  The
standard maximal torus $T_n(\overline{\F}_s)$ of
$\operatorname{Sp}_{2n}(\overline \F_s)$, described for instance in
\cite[Exercises 10.19, 10.29]{malle-testerman2011}, is maximally split
with respect to $F_q$ \cite[Definition~21.13,
Example~21.14]{malle-testerman2011} and  the Weyl group $W(C_n)$ of
$\operatorname{Sp}_{2n}(\overline \F_s)$ acts as the standard
representation of the signed permutation group $C_2 \wr \Sigma_n$ in
the $n$-dimensional real vector space
$X(T_n(\overline{\F}_s)) \otimes \R$ spanned by the character group
$X(T_n(\overline{\F}_s))$.  As usual, $T_n(\overline{\F}_s)_w$ denotes
the $F_q$-stable maximal torus of
$\operatorname{GL}_n(\overline \F_s)$ corresponding to the Weyl group
element $w \in W(C_n)$ \cite[Proposition~25.1]{malle-testerman2011}.
The number of elements in $T_n(\overline{\F}_s)_w$ that are fixed by
the Frobenius endo\m\ $F_q$ is $|T_n(\overline{\F}_s)_w^{F_q}| =
\det(q-w^{-1})$ where the determinant is computed in
$X(T_n(\overline{\F}_s)) \otimes \R$ \cite[Proposition~25.3.(c)]{malle-testerman2011}.

 \begin{proof}[Proof of Theorem~\ref{thm:WSp}]
   Conjugacy classes in $W(C_n) = C_2 \wr \Sigma_n$ are in bijective
   correspondence with partitions, $(\lambda_-, \lambda_+)$, of $n$
   into parts of two kinds \cite[Chapter I, Appendix B]{macdonald2015}
   \cite[Theorem~3.5]{tamanoi2003}.  If $w \in W(C_n)$ is in
   the conjugacy class of $(\lambda^-,\lambda^+)$ then (cf.\
   \cite[Example~25.4.(2)]{malle-testerman2011})
   \begin{equation*}
     \det(w^{-1}) \det(q-w^{-1})^r = (-1)^{n+|\lambda^-|}
     \prod_{d^- \in B(\lambda^-)} (q^{d^-}-1)^{rE(\lambda^-,d^-)}
     \prod_{d^+ \in B(\lambda^+)} (q^{d^+}+1)^{rE(\lambda^+,d^+)}
   \end{equation*}
  The claim of the theorem is thus that
  \begin{equation}\label{eq:thmWSp}
    -\chi_{r+1}(\Sp{2n}{q}) = 
    \sum_{(\lambda^-,\lambda^+) \vdash n} (-1)^{|\lambda^-|}
    2^{-|\lambda^-+\lambda^+|}
    \prod_{d^- \in B(\lambda^-)}
    \frac{(q^{d^-}-1)^{rE(\lambda^-,d^-)}}{|C_{d^-} \wr
      \Sigma_{E(\lambda^-,d^-)}|}
    \prod_{d^+ \in B(\lambda^+)}
    \frac{(q^{d^+}+1)^{rE(\lambda^+,d^+)}}{|C_{d^+} \wr \Sigma_{E(\lambda^+,d^+)}|}
  \end{equation}
  as the group $W(C_n)$ contains
  \begin{equation*}
    \frac{|C_2\wr \Sigma_n|}
    {\prod_{d^- \in B(\lambda^-)}|(C_2 \times C_{d^-}) \wr
      \Sigma_{E(\lambda^-,d^-)}|
    \prod_{d^+ \in B(\lambda^+)}|C_{2d^+} \wr
      \Sigma_{E(\lambda^+,d^+)}|}
  \end{equation*}
  elements in the conjugacy class $(\lambda^-,\lambda^+)$.

  By Corollary~\ref{cor:Bnq}, it suffices to show 
      \begin{equation*}
  \sum_{\substack{  (\lambda^-,\lambda^+) \vdash n \\  \lambda^-+\lambda^+=\lambda}}   (-1)^{|\lambda^+|} \prod_{d^- \in B(\lambda^-)}
  \frac{((q^{d^-}-1)^r)^{E(\lambda^-,d^-)}}{|C_{d^-} \wr \Sigma_{E(\lambda^-,d^-)}|}
  \prod_{d^+ \in B(\lambda^+)}
  \frac{((q^{d^+}+1)^r)^{E(\lambda^+,d^+)}}{|C_{d^+} \wr \Sigma_{E(\lambda^+,d^+)}|} =
\prod_{d \in B(\lambda)}
    \frac{((q^d-1)^r - (q^d+1)^r)^{E(\lambda,d)}}{|C_d \wr \Sigma_{E(\lambda,d)}|}
\end{equation*}
for all partitions $\lambda \vdash n$ and all integers $r \geq 1$.
  Introducing the the coefficients
\begin{equation*}
  c(\lambda^-,\lambda^+) = (-1)^{|\lambda^+|}
  \frac{\prod_{d \in B(\lambda)}|C_d \wr \Sigma_{E(\lambda,d)}|}
  {\prod_{d^- \in B(\lambda^-)}|C_{d^-} \wr \Sigma_{E(\lambda^-,d^-)}|
  \prod_{d^- \in B(\lambda^+)}|C_{d^+} \wr \Sigma_{E(\lambda^+,d^+)}|}
\end{equation*}
we need to show
\begin{equation*}
  \sum_{\lambda^-+\lambda^+=\lambda}  c(\lambda^-,\lambda^+)
  \prod_{d^- \in B(\lambda^-)} ((q^{d^-}-1)^r)^{E(\lambda^-,d^-)}
  \prod_{d^+ \in B(\lambda^+)} ((q^{d^+}+1)^r)^{E(\lambda^+,d^+)} 
  =
  \prod_{d \in B(\lambda)}
    ((q^d-1)^r - (q^d+1)^r)^{E(\lambda,d)}
  \end{equation*}
  That reason that this is true is that
  the binomial formula
  \begin{equation*}
    (a_1-b_1)^n = \sum_{1 \leq i \leq n} (-1)^{n-i} \frac{|C_1 \wr
        \Sigma_n|}{|C_1 \wr \Sigma_i| |C_1 \wr \Sigma_{n-i}|} a_1^i b_1^{n-i}
  \end{equation*}
  generalizes to the identity
  \begin{equation*}
  \prod_{d \in B(\lambda)}
    (a_d-b_d)^{E(\lambda,d)} =
  \sum_{\lambda^-+\lambda^+=\lambda}  c(\lambda^-,\lambda^+)
   \prod_{d^- \in B(\lambda^-)} a_{d^-}^{E(\lambda^-,d^-)}
  \prod_{d^+ \in B(\lambda^+)} b_{d^+}^{E(\lambda^+,d^+)} 
  \end{equation*}
 in the \pol\ ring 
 $\Z[a_d,b_d \mid d \in
  B(\lambda)]$ with the $2|B(\lambda)|$ indeterminates $a_d,b_d$, $d \in B(\lambda)$.
 \end{proof}

 The right hand side of the identity from Theorem~\ref{thm:WSp} is
 \begin{multline*}
   \frac{(-1)^n}{|W(C_n)|}\sum_{w \in W(C_n)} \det(w) \det(q-w)^r =
   \frac{(-1)^n}{|W(C_n)|}\sum_{w \in W(C_n)} \det(w^{-1}) \det(q-w^{-1})^r \\ =
  \frac{(-1)^n}{|W(C_n)|}\sum_{w \in W(C_n)} \det(w) \det(q-w^{-1})^r =                                                        
  \frac{(-1)^n}{|W(C_n)|}\sum_{w \in W(C_n)} \det(w) |T_n(\overline{\F}_s)_w^{F_q}|^r 
\end{multline*}
where we used  \cite[Proposition~25.3(c)]{malle-testerman2011} and  $\det(w)=\det(w^{-1})$.

\begin{cor} \label{cor:WSp}  
The generating functions for the sequences
 $(-\rchi_{r+1}(\Sp{2n}q))_{r \geq 0}$ and
 $(-\rchi_{r+1}(p,\Sp{2n}q))_{r \geq 0}$ (with fixed $n$) are
 \begin{align*}
   \sum_{r \geq 0} -\rchi_{r+1}(\Sp{2n}q) x^r &=
  \frac{(-1)^n}{|W(C_n)|} \sum_{w \in W(C_n)}
                                                \frac{\det(w)}{1-x\det(q-w)} \\
   \sum_{r \geq 0} -\rchi_{r+1}(p,\Sp{2n}q) x^r &=
  \frac{(-1)^n}{|W(C_n)|} \sum_{w \in W(C_n)}
  \frac{\det(w)}{1-x\det(q-w)}_p
 \end{align*}
for all $n \geq 1$.
\end{cor}

By considering conjugacy classes rather than the individual elements in $W(C_n)$,
the formulas of Corollary~\ref{cor:WSp} can also be written as
\begin{align}\label{eq:corWSp}
  \sum_{r \geq 0} -\rchi_{r+1}(\Sp{2n}q) x^r &=
  \sum_{(\lambda^-,\lambda^+) \vdash n} (-1)^{|\lambda^-|} \frac{1}{T(\lambda^-,\lambda^+)}
                                               \frac{1}{\displaystyle 1-x U(\lambda^-,\lambda^+)} \\
  \sum_{r \geq 0} -\rchi_{r+1}(p,\Sp{2n}q) x^r &=
  \sum_{(\lambda^-,\lambda^+) \vdash n} (-1)^{|\lambda^-|} \frac{1}{T(\lambda^-,\lambda^+)}
  \frac{1}{\displaystyle 1-x U(\lambda^-,\lambda^+)_p} \label{eq:corWSpB}
\end{align}
where
\begin{align*}
  &T(\lambda^-,\lambda^+) =\prod_{d^- \in B(\lambda^-)} |C_{d^-} \wr \Sigma_{E(\lambda^-,d^-)}|
    \prod_{d^- \in B(\lambda^-)} |C_{d^-} \wr
    \Sigma_{E(\lambda^-,d^-)}| \\
    &U(\lambda^-,\lambda^+) =\prod_{d^- \in B(\lambda^-)} (q^{d^-}-1)^{E(\lambda^-,d^-)}
    \prod_{d^+ \in B(\lambda^+)} (q^{d^+}+1)^{E(\lambda^+,d^+)}
\end{align*}
for every partition  $(\lambda^-,\lambda^+)$ of $n$ into parts of two
kinds.

\begin{cor}\label{cor:Spinv}
  Let $\rho \colon W(C_n)=  C_2 \wr \Sigma_n\to W(A_n)= \Sigma_n$
  denote the projection with kernel $C_2^n$. Then
  \begin{align*}
    1 + \sum_{n \geq 1}  \frac{x^n}{|W(C_n)|} \sum_{w \in W(C_n)} \det(\rho(w)) \det(w)
    \det(q-w)^r  &=
    \prod_{\substack{0 \leq j \leq r \\ j \not\equiv r
    \bmod 2}} (1-q^jx)^{\binom{r}{j}} \\
    1 + \sum_{n \geq 1}  \frac{x^n}{|W(C_n)|} \sum_{w \in W(C_n)} \det(\rho(w)) \det(w)
    \det(q-w)^r_p &= \FSp{r+1}(p,q,x)^{-1}
  \end{align*}
\end{cor}
\begin{proof}
   We find expresssions for the reciprocal power series
   $\FSp{r+1}(q,x)^{-1}$ and $\FSp{r+1}(p,q,x)^{-1}$.
   As in Corollary~\ref{cor:Bnq}, we have
  \begin{equation*}
    \FSp {r+1}(q,x)^{-1} =
    \exp\big(\sum_{n \geq 1} H(r,q^n) \frac{x^n}{n} \big) =
    \sum_{\lambda \vdash n} x^n
    \prod_{d \in B(\lambda)}
                                 \frac{((q^d-1)^r/2-(q^d+1)^r/2)^{E(\lambda,d)}}{|C_d
                                 \wr \Sigma_{E(\lambda,d)}|}
                             \end{equation*}
     and we can identify the coefficients of this power series as sums
     indexed by $W(C_n)$ as in the proof of Theorem~\ref{thm:WSp}.                   
\end{proof}

\begin{exmp}\label{exmp:WSp}
Corollary~\ref{cor:Bnq} for $n=1,2,3$ shows
\begin{equation*}  
  -\rchi_{r+1}(\Sp {2n}q) =
  \begin{cases}
    -H(r,q) & n=1 \\
    \frac{1}{2}H(r,q)^2 - \frac{1}{2}H(r,q^2) & n=2 \\
     -\frac{1}{6}H(r,q)^3 +
                         \frac{1}{2}H(r,q)H(r,q^2) -
                         \frac{1}{3}H(r,q^3) & n=3
  \end{cases}
\end{equation*}
where $H(r,q) = \frac{1}{2}(q-1)^r - \frac{1}{2}(q+1)^r$.
With fixed $n=1,2$, Theorem~\ref{thm:WSp} (in the formulation of \eqref{eq:thmWSp}) shows that
\begin{equation*}
  -\rchi_{r+1}(\Sp{2n}q) =
  \begin{cases}
    \frac{1}{2}(q+1)^r - \frac{1}{2}(q-1)^r & n=1 \\
  \frac{1}{8}(q-1)^{2r} + \frac{1}{8}(q+1)^{2r} +
    \frac{1}{4}(q^2+1)^r   - \frac{1}{4}(q^2-1)^r -
    \frac{1}{4}(q-1)^r(q+1)^r & n=2
  \end{cases}
\end{equation*}
and with fixed $r=1,2,3$ it shows that
\begin{equation*}
\frac{(-1)^n}{|W(C_n)|}\sum_{w
  \in W(C_n)} \det(w) \det(q-w)^r =
\begin{cases}
  1 & r=1\\
  nq^{n-1} & r=2 \\
  \sum_{0 \leq j \leq n} \binom{2+j}{2}q^{2j} & r=3
\end{cases}
\end{equation*}
for all $n \geq 1$. From
Corollary~\ref{cor:WSp}  (in the formulation
  of \eqref{eq:corWSp}) for $n=1,2$  we get
  \begin{equation*}  
    \sum_{r \geq 0} -\rchi_{r+1}(\Sp{2n}q) x^r =
    \begin{cases}
       \frac{\frac{1}{2}}{1-(q+1)x} -
                                              \frac{\frac{1}{2}}{1-(q-1)x}
                                              &  n=1 \\
                   \frac{\frac{1}{8}}{1-x(q-1)^2} +
                                              \frac{\frac{1}{8}}{1-x(q+1)^2}
                                              +
                                              \frac{\frac{1}{4}}{1-x(q^2+1)}
                                              -
                                              \frac{\frac{1}{4}}{1-x(q^2-1)}
                                              -
                                              \frac{\frac{1}{4}}{1-x(q-1)(q+1)}
                                              & n=2                             
    \end{cases}
  \end{equation*}
  In the $p$-primary case,
  when $p=2$ and $q= 3^{2^e}$ with $e>0$, $(q-1)_2=2^{2+e}$,
$(q+1)_2=2$, $(q^2-1)_2=2^{3+e}$, $(q-1)_2(q+1)_2=2^{3+e}$,
$(q^2+1)_2=2$,    and we get
\begin{align*}
  \sum_{r \geq 0} -\rchi_{r+1}(2,\Sp{2}{3^{2^e}}) x^r &=
                                              \frac{\frac{1}{2}}{1-2x} -
                                              \frac{\frac{1}{2}}{1-2^{2+e}x} \\
  \sum_{r \geq 0} -\rchi_{r+1}(2,\Sp{4}{3^{2^e}}) x^r &=
                                              \frac{\frac{1}{8}}{1-4^{2+e}x} +
                                              \frac{\frac{1}{8}}{1-4x}
                                              +
                                              \frac{\frac{1}{4}}{1-2x}
                                              -
                                              \frac{\frac{1}{4}}{1-2^{3+e}x}
                                              -
                                              \frac{\frac{1}{4}}{1-2^{3+e}x}
\end{align*}
from Corollary~\ref{cor:WSp} (in the formulation of
\eqref{eq:corWSpB}).
Corollary~\ref{cor:Spinv} with $r=1,2$ and
Example~\ref{exmp:FSpprim}  show that
\begin{align*}
  &1 + \sum_{n \geq 1}  \frac{x^n}{|W(C_n)|} \sum_{w \in W(C_n)} \det(\rho(w)) \det(w)
  \det(q-w)^r =
  \begin{cases}
    1-x & r=1 \\ 1- 2qx + x^2 & r=2
  \end{cases} \\
       &1 + \sum_{n \geq 1}  \frac{x^n}{|W(C_n)|} \sum_{w \in W(C_n)} \det(\rho(w)) \det(w)
  \det(3-w)_2^r =
  \begin{cases}
    \displaystyle \frac{1-x}{\prod_{n \geq 0} Q(x^{2^{n+1}})} & r=1 \\
     \displaystyle \frac{(1-x)^2Q(x)^2}{\prod_{n \geq 0} Q(x^{2^{n+1}})^{2^{n+3}}} & r=2
  \end{cases}
\end{align*}
where $Q(x) = \frac{1-x}{1+x}$. Consequently,
 $\sum_{w \in W(C_n)} \det(\rho(w)) \det(w)
    \det(q-w)^r=0$ for all $n>r$ if $r=1,2$.
\end{exmp}

\section{Polynomial identities for partitions into parts of two kinds}
\label{sec:polyn-ident}

For any \pol\ sequence $S$ and any rational number $m$ the
$mS$-transform
of $1 \pm x$ is \cite[Lemma~7.1]{jmm:eulergl-} 
\begin{equation*}
  T_{mS}(1 \pm x) =
  \prod_{d \geq 1} (1 \pm x^d)^{mS(d)(q)} =
  \prod_{d \geq 1}  \sum_E \binom{mS(d)(q)}{\pm E} x^{dE}=
  \sum_{n \geq 0} x^n
    \sum_{\lambda \vdash n} \prod_{d \in B(\lambda)}
    \binom{m  S(d)(q)}{\pm E(\lambda,d)} 
  \end{equation*}
The classical \pol\ identity $T_{\IM {}q}(1-x) = \frac{1-qx}{1-x}$
gives  the \pol\ identity
 \begin{equation*}
    \sum_{n \geq 0} x^n
    \sum_{\lambda  \vdash n} \prod_{d \in B(\lambda)}
    \binom{m\IM dq}{-E(\lambda,d)} 
    = \left( \frac{1-qx}{1-x} \right)^m
  \end{equation*}
  for partitions. The cases $m = \pm 1$ are
Th{\'e}venaz' \pol\
identities \cite[Theorems~A--B]{thevenaz92poly} \cite[Corollary~7.2]{jmm:eulergl-}.

The  identities
$  T_{-\SRIM -{}q}(1+x) T_{-\SRIM +{}q}(1-x) = 1-x$ and
$  T_{\SRIM -{}q}(1-x) T_{-\SRIM -{}q}(1+x) =   \frac{1-qx}{1-x}$ (for
odd $q$)
of Lemma~\ref{lemma:ofir} 
translate into the following \pol\ identities
\begin{align*}  
  \sum_{n \geq 0} x^n \sum_{(\lambda^-,\lambda^+) \vdash n}
  \prod_{d^- \in B(\lambda^-)} \binom{-m \SRIM
  -{d^-}q}{E(\lambda^-,d^-)}
  \prod_{d^+ \in B(\lambda^+)} \binom{-m \SRIM
  +{d^+}q}{-E(\lambda^+,d^+)}  &= (1-x)^m \\
  \sum_{n \geq 0} x^n \sum_{(\lambda^-,\lambda^+) \vdash n}
  \prod_{d^- \in B(\lambda^-)} \binom{m \SRIM
  -{d^-}q}{-E(\lambda^-,d^-)}
  \prod_{d^+ \in B(\lambda^+)} \binom{-m \SRIM
    -{d^+}q}{E(\lambda^+,d^+)}  &=
                                  \begin{cases}
                                   \displaystyle \Big( \frac{1-qx}{1-x} \Big)^m & 2
                                   \nmid q \\
                                   (1-qx)^m &  2 \mid q
                                  \end{cases}
\end{align*}
for partitions into parts of two kinds.

\begin{exmp}
  Based on the partitions,
$\{(1^1,\emptyset), (\emptyset,1^1)\}$ and
$\{(2^1, \emptyset), (1^2, \emptyset),
(1^1,1^1), (\emptyset,1^2), (\emptyset,2^1)\}$, of $1$ and $2$ 
into parts of two kinds, we  have
the identities, valid for any rational number $m$,
\begin{gather*}
  \binom{m\SRIM +1q}{-1} +
  \binom{m\SRIM -1q}{1} = - \binom{-m}{1} \\
  \binom{m\SRIM +2q}{-1} +
  \binom{m\SRIM +1q}{-2} +
  \binom{m\SRIM -1q}{1} \binom{m\SRIM +1q}{-1} +
  \binom{m\SRIM -2q}{1} +
  \binom{m\SRIM -1q}{2}= \binom{-m}{2} \\
  \binom{m\SRIM -1q}{-1} +
  \binom{-m\SRIM -1q}{1} =
  \begin{cases}
   -mq+m & 2 \nmid q \\ -mq & 2 \mid q    
  \end{cases}
\\
  \binom{m\SRIM -2q}{-1} +
  \binom{m\SRIM -1q}{-2} +
  \binom{m\SRIM -1q}{-1} \binom{-m\SRIM -1q}{1} +
  \binom{-m\SRIM -1q}{2} +
  \binom{-m\SRIM -2q}{1} \\=
  \begin{cases}
  \binom{m}{2}q^2 - m^2q + \binom{m+1}{2} & 2 \nmid q \\ \binom{m}{2}q^2 & 2 \mid q  
  \end{cases}
\end{gather*}
 by comparing coefficients of $x^n$ for $n=1,2$.
\end{exmp}

\appendix

 \section{Equivariant \Euc s of posets}
 \label{sec:equivariant-euc-s-1}

This appendix contains a  few elementary observations about
equivariant \Euc s for group actions on posets.
 
 Let $S$ be a finite set and $\dim \colon S \to \Z$ a
 function associating an integer $\geq -1$ to every element of $S$.
 The
 \Euc\ and the reduced \Euc\ of the graded set $(S,\dim)$ are the alternating sums
 \begin{equation*}
   \chi(S,\dim) = \sum_{d \geq 0} (-1)^d|\dim^{-1}(d)| \qquad \rchi(S,\dim) = \sum_{d
    \geq -1} (-1)^d|\dim^{-1}(d)| =
  \chi(S,\dim)-|\dim^{-1}(-1)|
 \end{equation*}
 of the numbers of $d$-dimensional elements of $S$ for $d \geq 0$ or
 $d \geq -1$.

 Let $\Pi$ be a finite poset.  A simplex in $\Pi$ is a totally ordered
 subset of $\Pi$.  The set $|\Pi|$ of all simplices in $\Pi$ is graded
 by the function $\dim \colon |\Pi| \to \Z$ taking a simplex
 $\sigma \subseteq \Pi$ to one less than its cardinality,
 $\dim \sigma = |\sigma|-1$.

\begin{defn}\label{defn:Eucposeti}
The \Euc\ of the poset $\Pi$ is $\chi(\Pi) =
\chi(|\Pi|,\dim)$ and the reduced \Euc\  is  $\rchi(\Pi) =
\rchi(|\Pi|,\dim) =\chi(\Pi)-1$.
\end{defn}

Let $G$ be a finite group. Write $\Hom{\Z^r}G$ for the set of
homo\m s of $\Z^r$ to $G$ and 
$\Hom{\Z^r}G/G$ for the set of conjugacy classes of
such homo\m s. Equivalently, $\Hom{\Z^r}G$ is the set of commuting
$r$-tuples of elements in $G$ and $\Hom{\Z^r}G/G$ is the set of
conjugacy classes of commuting $r$-tuples.

Suppose now that $G$ acts on the poset $\Pi$ through order preserving
bijections.  For any finite subset $X$ of $G$, let
$C_\Pi(X)=\{p \in \Pi \mid \forall g \in X \colon p^g = p\}$ denote
the full subposet of elements of $\Pi$ fixed under the action from
$X$.

\begin{defn}[Atiyah and Segal \cite{atiyah&segal89}]\label{defn:atiyahsegal}
The $r$th, $r \geq 1$, equivariant \Euc\ and
reduced equivariant \Euc\ of the $G$-poset $\Pi$ are 
\begin{align*}
  \chi_r(\Pi,G) &= \frac{1}{|G|} \sum_ {X \in \Hom {\Z^r}{G}}
  \chi(C_\Pi(X(\Z^r))) \\ 
  \rchi_r(\Pi,G) &= \frac{1}{|G|} \sum_ {X \in \Hom {\Z^r}{G}}
  \rchi(C_\Pi(X(\Z^r))) = \chi_r(\Pi,G) - |\Hom{\Z^r}G|/|G|
\end{align*}  
\end{defn}
The equivariant \Euc s satisfy a recurrence relation.

\begin{lemma}\label{lemma:CgP}
  For all $r \geq 1$,
  \begin{equation*}
    \rchi_{r+1}(\Pi,G) =
    \sum_{X \in \Hom{\Z}G/G} \rchi_r(C_\Pi(X),C_G(X)) =
    \sum_{X \in \Hom{\Z^r}G/G} \rchi_1(C_\Pi(X),C_G(X))
  \end{equation*}
  and similar formulas are true for $\chi_{r+1}(\Pi,G)$.
\end{lemma}
\begin{proof}
  A little more generally, we consider $\rchi_{r_1+r_2}(\Pi,G)$ for $r_1
  \geq 1$ and $r_2 \geq 2$. The equivariant \Euc\ is
  \begin{multline*}
    \rchi_{r_1+r_2}(\Pi,G) =
    \frac{1}{|G|} \sum_{X \in \Hom {\Z^{r_1+r_2}}{G}}
    \rchi(C_\Pi(X)) 
      =\frac{1}{|G|} \sum_{X_1  \in \Hom{\Z^{r_1}}G}\sum_ {X_2 \in
        \Hom {\Z^{r_2}}{C_G(X_1)}} \rchi(C_{C_\Pi(X_1)}(X_2)) \\
    = \frac{1}{|G|} \sum_{X_1  \in \Hom{\Z^{r_1}}G} |C_G(X_1)|
    \rchi_{r_2}(C_\Pi(X_1),C_G(X_1)) =
    \sum_{X_1  \in \Hom{\Z^{r_1}}G/G}   \rchi_{r_2}(C_\Pi(X_1),C_G(X_1)) 
  \end{multline*}
  where we use that the conjugacy class of $X_1$ contains $|G : C_G(X_1)|$ elements.
\end{proof}

The set $|C_\Pi(X)|/C_G(X)$ of $C_G(X)$-orbits of
$C_\Pi(X)$-simplices,
for any $X \subseteq G$,
has \Euc\ relative to the dimension function induced by the dimension
function on $|\Pi|$.

\begin{lemma}\label{lemma:chi1}
 $\displaystyle \rchi_{r+1}(\Pi,G) = \sum_{X  \in \Hom{\Z^{r}}G/G}
 \rchi(|C_\Pi(X)|/C_G(X))$ for all $r \geq 0$.
\end{lemma}
\begin{proof}
  We first consider the case $r=0$.
  The orbit counting formula shows that
  \begin{align*}
    \rchi(|\Pi|/G) = \sum_{d \geq -1} (-1)^d | \dim^{-1}(d)/G |& =
    \frac{1}{|G|}
    \sum_{d \geq -1} (-1)^d \sum_{g \in G} ||C_\Pi(g)| \cap \dim^{-1}(d)|
    \\ &=
    \frac{1}{|G|} \sum_{g \in G} \sum_{d \geq -1} (-1)^d ||C_\Pi(g)| \cap \dim^{-1}(d) |  = \frac{1}{|G|}
    \sum_{g \in G} \chi(C_\Pi(g)) = \chi_1(\Pi,G)
  \end{align*}
  Consequently, for all $r \geq 1$,
  \begin{equation*}
    \rchi_{r+1}(\Pi,G)
  =  \sum_{X  \in \Hom{\Z^{r}}G/G} \rchi_1(C_\Pi(X),C_G(X)) =  \sum_{X  \in \Hom{\Z^{r}}G/G} \rchi(|C_\Pi(X)|/C_G(X))
  \end{equation*}
  by Lemma~\ref{lemma:CgP}.
\end{proof}
It is clear from Lemma~\ref{lemma:chi1}, but maybe not from
Definition~\ref{defn:atiyahsegal},
that all equivariant \Euc s are integers.

\section{Eulerian functions of groups}
\label{sec:comp-equiv-euc}

Let $G$ be a finite group acting on a finite poset $\Pi$. For any
natural number $r \geq 1$, the $r$th equivariant reduced \Euc\ and the
$p$-primary $r$th equivariant reduced \Euc\ are \cite{atiyah&segal89} \cite{tamanoi2001}
\begin{align}
  &\rchi_r(\Pi,G) = \frac{1}{|G|}\sum_{X \in \Hom{\Z^r}G} \rchi(C_{\Pi}(X)) =
  \frac{1}{|G|} \sum_{ B \leq G} \varphi_{\Z^r}(B) \rchi(C_\Pi(B)) \label{rchirPiG:item1}\\
  &\rchi_r(p,\Pi,G) = \frac{1}{|G|} \sum_{X \in \Hom{Z_p^r}G} \rchi(C_\Pi(X)) =
  \frac{1}{|G|}\sum_{ B \leq G} \varphi_{Z_p^r}(B) \rchi(C_\Pi(B)) \label{rchirPiG:item2}
\end{align}
where $\varphi_{\Z^r}(B)$ ($\varphi_{Z_p^r}(B)$) is the number of
epi\m s of the abelian group $\Z^r$ ($Z_p^r =\Z \times \Z_p^{r-1}$) onto the
subgroup $B$ of $G$. In this appendix, we recall some of the
properties, helpful for concrete computer assisted calculations of
equivariant \Euc s, of the eulerian function $\varphi_{\Z^r}(B)$
\cite{hall36}.


For any finite group $B$, let $\Hom{\Z^r}B$ and $\Epi{\Z^r}B$ be the
set of homo\m s or epi\m s of $\Z^r$ to $B$. Then
$\Hom{\Z^r}B =\coprod_{A \leq B} \Epi{\Z^r}A$ and
$\varphi_{\Z^r}(B) = |\mathrm{Epi}(\Z^r,B)|$. 
(When $r=1$ and $C_n$ is
cyclic of order $n$, $\varphi_{\Z^1}(C_n)$ is Euler's totient function
$\varphi(n)$.)
We observe that $\varphi_{\Z^r}$ is multiplicative.

\begin{lemma}\label{lemma:phir}
  Let $B_1$ and $B_2$ be two finite groups of coprime order.
  \begin{enumerate}
  \item \label{item:phir1}  For any subgroup
  $A$ of $B_1 \times B_2$, $A = A_1 \times A_2$ where $A_i$ is the image of
  $A$ under the projection $B_1 \times B_2 \to B_i$, $i=1,2$.
\item \label{item:phir2} $\varphi_{\Z^r}(B_1 \times B_2) = \varphi_{\Z^r}(B_1) \times
  \varphi_{\Z^r}(B_2)$ for any $r \geq 1$
  \end{enumerate}
\end{lemma}
\begin{proof}
  Let $g_i$ be the order of $B_i$, $i=1,2$. The order of $A$, which divides
  $g_1g_2$, is of the form $k_1k_2$ where $k_1$ divides $g_1$ and $k_2$
  divides $g_2$. The order of $A_i$ divides $k_1k_2$ and $g_i$. Thus $|A_i|$
  divides $k_i$.  It follows that the order of $A_1 \times A_2$ divides the
  order of $A$. But $A$ is a subgroup of $A_1 \times A_2$ so $|A| = |A_1
  \times A_2|$ and $A= A_1 \times A_2$.
\end{proof}

Next, we compute $\varphi_{\Z^r}(C_p^d)$ where $C_p^d$ is elementary
abelian of order $p^d$.  First,
$\Epi{\Z^r}{C_p^d} = \Epi{C_p^r}{C_p^d}$, the set of epi\m s of
$C_p^r$ onto $C_p^d$.  Next, note that there is a bijection between
the orbit set $\Epi{C_p^r}{C_p^d}/\Aut{}{C_p^d}$ and the set of
$(r-d)$-dimensional subspaces of $\F_p^r$ (kernels of epi\m s). The
number of such subspaces is the Gaussian binomial coefficient
$
  \binom r{r-d}_p = \binom rd_p
$ \cite[Proposition~1.3.18]{stanley97}.
Thus
\begin{equation}\label{eq:epiprpd}
  \Epi{\Z^r}{C_p^d} = \binom rd_p |\GL d{p}| = \prod_{j=0}^{d-1} (p^r-p^j)
\end{equation}
In the general case,  the number of homo\m\ of $\Z^r$ to
$B$ is
\begin{equation*}
  |\Hom{\Z^r}B| = \sum_{A \leq B}|\Epi{\Z^r}A|
   = \sum_{A \leq G}|\Epi{\Z^r}A| \zeta(A,B)
 \end{equation*}
 where $\zeta(A,B)=1$ if $A \leq B$ and $\zeta(A,B)=0$ otherwise.
The number of epi\m\ of $\Z^r$ onto $B$ is
\begin{equation*}
  \varphi_{\Z^r}(B) =
  |\Epi{\Z^r}B| = \sum_{A \leq G}|\Hom{\Z^r}A| \mu(A,B)
\end{equation*}
by \Mb\ inversion. Of course, $\varphi_{\Z^r}(B)>0$ if and only if $B$
is abelian and generated by $r$ of its elements. Assuming $B$ {\em is\/} abelian,
$|\Hom{\Z^r}A|=|A|^r$ for any $A \leq B$ so that
\cite{hall36,gaschutz59,wall62}
\begin{equation*}
   \varphi_{\Z^r}(B) = 
  |\Epi{\Z^r}B| = \sum_{A \leq B} |A|^r \mu(A,B)
\end{equation*}
The \Mb\ function $\mu(A,B)=0$ unless $\Phi(B) \leq A \leq B$ and then
$\mu_B(A,B) = \mu_{B/\Phi(B)}(A/\Phi(B),B/\Phi(B))$ where $\Phi(B)$ is
the Frattini subgroup \cite{gaschutz59}. Therefore
\begin{equation*}
   \varphi_{\Z^r}(B) 
   = \sum_{A \leq B} |A|^r \mu_B(A,B) =
   |\Phi(B)|^r  \sum_{A \leq B/\Phi(B)} |A|^r \mu_{B/\Phi(B)}(A,B/\Phi(B)) =
   |\Phi(B)|^r \varphi_{\Z^r}(B/\Phi(B))
 \end{equation*}
 The abelian group $B$ is the product, $B=\prod B_p$, of its \syl ps, $B_p$. By
multiplicativity (Lemma~\ref{lemma:phir}.\eqref{item:phir2}),
\begin{equation*}
  \varphi_{\Z^r}(B) = \prod_p \varphi_{\Z^r}(B_p)
\end{equation*}
The Frattini quotient $B_p/\Phi(B_p)$ is an elementary abelian $p$-group of
order, say, $p^d$. We conclude that
\begin{equation*}
   \varphi_{\Z^r}(B_p) = 
    |\Phi(B_p)|^r |\Epi{\Z^r}{C_p^d}| \stackrel{\text{\eqref{eq:epiprpd}}}{=}
    |\Phi(B_p)|^r \prod_{j=0}^{d-1} (p^r-p^j) =
    |B_p|^r \prod_{j=0}^{d-1} (1-p^{j-r}) 
  \end{equation*}
  For the final equality, use that
if $B_p$ has order $p^m$, then the order of the Frattini subgroup is $p^{m-d}$
so that $|\Phi(B_p)|^r = p^{r(m-d)}$.


For a prime $p$, put $Z^r_p = \Z \times \Z_p^{r-1}$ where $\Z_p$ is
the ring of $p$-adic integers. In particular, $Z^1_p = \Z$ is
independent of $p$. The number of epi\m s of $Z_p^r$ onto $B$ is
\begin{equation*}
  \varphi_{Z_p^r}(B) = \varphi_{Z_p^r} (\prod_s B_s) = \prod_s
  \varphi_{Z_p^r}(B_s) = \varphi_{\Z^r}(B_p) \prod_{s \neq p} \varphi_{\Z}(B_s)
\end{equation*}
where $B_s$ is the \syl s\ of $B$. Here,
$\varphi_{\Z}(B_s)=|B_s|(1-p^{-1})$ if $B_s$ is cyclic and
$\varphi_{\Z}(B_s)=0$ otherwise. Thus $\varphi_{Z_q^r}(B)>0$ if and
only if $B_q$ can be generated by $r$ of its elements and $B_s$ is
cyclic for all primes $s \neq q$.

The numbers of conjugacy classes of $r$-tuples of commuting
elements of $G$ and commuting $p$-power order elements are
\begin{equation*}
  |\Hom{\Z^r}G/G| = |\Hom{\Z^{r+1}}G|/|G|, \qquad
  |\Hom{\Z_p^r}G/G| = |\Hom{Z_p^{r+1}}G|/|G| \qquad (r \geq 0)
\end{equation*}
as $|\Hom{K}G/G| = |\Hom{\Z \times K}G|/|G|$ for any finite group $K$
\cite[Lemma~4.13]{HKR2000}.

\begin{exmp}
  The symplectic group $G=\Sp 23$, of order $24$, acts on  the discrete
  poset $L=\Li_2^*(\F_3)$ of $4$
  elements. Using Equation~\eqref{rchirPiG:item1} and the entries of
  the 
  table in Figure~\ref{fig:chirsp23} 
\begin{figure}[t]
  \centering
    \begin{tabular}[h]{>{$}c<{$}|*{3}{>{$}c<{$}}}
 A & C_2^0 & C_2^1  & C_2^2  \\ \hline
      |G : C_G(A)| & 1 & 1& 3  \\
   \rchi(C_L(A)) & 3& 3& -1 \\
\varphi_{\Z^{r+1}}(A) & 1 & 2^{r+1}-1 & 4^{r+1}-2^{r+1}
    \end{tabular}
  \caption{Abelian $3'$-subgroups of $\Sp 23$ }
  \label{fig:chirsp23}
\end{figure}
we get that
\begin{equation*}
  -\rchi_{r+1}(\Li_2^*(\F_3), \Sp 23) =
  -\frac{1}{24}(3+3(2^{r+1}-1) - 3(4^{r+1}-2^{r+1})) = \frac{1}{2}(4^r-2^r)
\end{equation*}
in accordance with Example~\ref{exmp:WSp}. By
Lemma~\ref{lemma:contrability} we only need to consider abelian
subgroups of $\Sp 23$ of order prime to $3$.
\end{exmp}

  \bibliographystyle{amsplain}
\bibliography{/Users/jespermoller/projects/top}
\end{document}